\documentclass[12pt]{article}
\usepackage{amsmath,amsfonts,amssymb,amsthm}
\usepackage{epsfig}
\usepackage{comment}

\theoremstyle{plain}
\newtheorem{theorem}{Theorem}[section]
\newtheorem{corollary}[theorem]{Corollary}
\newtheorem{lemma}[theorem]{Lemma}
\newtheorem{proposition}[theorem]{Proposition}

\theoremstyle{definition}
\newtheorem{definition}[theorem]{Definition}

\newtheorem{example}[theorem]{Example}
\newtheorem{remark}[theorem]{Remark}

\numberwithin{equation}{section}

\newcommand{\Z}{{\mathbb Z}}
\newcommand{\R}{{\mathbb R}}
\newcommand{\N}{{\mathbb N}}
\newcommand{\C}{{\mathbb C}}

\newcommand{\bdy}{\partial\Omega}
\newcommand{\Om}{\Omega}
\newcommand{\cOm}{\overline{\Omega}}
\providecommand{\vint}[1]{\mathchoice
          {\mathop{\vrule width 5pt height 3 pt depth -2.5pt
                  \kern -9pt \kern 1pt\intop}\nolimits_{\kern -5pt{#1}}}
          {\mathop{\vrule width 5pt height 3 pt depth -2.6pt
                  \kern -6pt \intop}\nolimits_{\kern -3pt{#1}}}
          {\mathop{\vrule width 5pt height 3 pt depth -2.6pt
                  \kern -6pt \intop}\nolimits_{\kern -3pt{#1}}}
          {\mathop{\vrule width 5pt height 3 pt depth -2.6pt
                  \kern -6pt \intop}\nolimits_{\kern -3pt{#1}}}}

\newcommand{\eps}{\varepsilon}
\newcommand{\loc}{{\mbox{\scriptsize{loc}}}}

\newcommand{\cH}{\mathcal{H}}

\newcommand{\BV}{\mathrm{BV}}
\newcommand{\liploc}{\mathrm{Lip}_{\mathrm{loc}}}

\DeclareMathOperator{\capa}{Cap}
\DeclareMathOperator{\rcapa}{cap}

\DeclareMathOperator{\dist}{dist}
\DeclareMathOperator{\diam}{diam}

\DeclareMathOperator{\Lip}{Lip}

\DeclareMathOperator{\supp}{supp}

\DeclareMathOperator{\Arg}{Arg}

\def\XXint#1#2#3{{\setbox0=\hbox{$#1{#2#3}{\int}$}
\vcenter{\hbox{$#2#3$}}\kern-.5\wd0}}

\begin{document}
\title{Trace theorems for functions of bounded variation in metric spaces 
\thanks{
\hfill \break {\it Keywords\,}: BV functions, bounded variation, jump sets, traces, Whitney cover, discrete convolution,
capacitary inequality.}
}
\author{P.~Lahti\thanks{Department of Mathematics and Systems Analysis, P.O. Box 11100, FI-00076 Aalto University, Finland. 
E-mail: {\tt panu.lahti@aalto.fi}.} and N.~Shanmugalingam\thanks{Department of Mathematical Sciences, P.O. Box 210025, University of
Cincinnati, Cincinnati, OH 45221--0025, U.S.A. 
E-mail: {\tt shanmun@uc.edu}}}
\maketitle

\begin{abstract} 
In this paper we show existence of traces of functions of bounded variation on the boundary of a certain class of domains in
metric measure spaces equipped with a doubling measure supporting a $1$-Poincar\'e inequality, and obtain
$L^1$ estimates of the trace functions. In contrast with 
the treatment of traces given in other papers on this subject, the traces we consider
do not require knowledge of the function in the exterior of the domain. 
We also establish a Maz'ya-type inequality for functions of bounded
variation that vanish on a set of positive capacity.
\end{abstract}

\section{Introduction}

The Dirichlet problem for functions of least gradient on a domain $\Omega$ is to find a function $u$ that minimizes 
the energy $\int_\Omega |\nabla u|$ amongst the class of all Sobolev functions with 
prescribed boundary values on $\partial\Omega$.
In order to make sense of the problem, one needs to know how to extend the function $u$ (which is a priori defined only on
$\Omega$) to $\partial\Omega$, and one needs to know for which boundary data the problem 
is solvable. The focus of the current paper
is to study the issue of how a function in the class $\BV(\Omega)$ can be extended to $\partial\Omega$, 
that is, whether it has
a \emph{trace} on $\partial\Omega$; the second question needs very strong geometric conditions on 
$\Omega$, even in the Euclidean
setting (see for example~\cite{Giu84}), and will not be addressed in the present paper.

In the classical Euclidean setting, a standard result is that if the boundary $\partial\Omega$ can be presented locally as 
a Lipschitz graph, then the trace of a $\BV$ function exists.
Classical treatments of boundary traces of $\BV$ functions can be found in 
e.g.~\cite[Chapter 3]{AFP}, ~\cite[Chapter 2]{Giu84}, and~\cite{AnzGia},
and similar results for Carnot groups are given in \cite{Vit}.

In the setting of general metric measure spaces, where the standard assumptions are a doubling measure and a Poincar\'e 
inequality, traces of BV functions have only been studied recently. In~\cite{HKLL}, results on traces are obtained by 
making rather strong assumptions on the extendability of BV functions from the domain to the whole space. These extendability 
properties are satisfied by \emph{uniform} domains, which are a natural generalization of Euclidean domains with Lipschitz 
boundaries. In~\cite{L}, results on traces are obtained in more general domains but with 
stronger assumptions on the metric space.
In these papers, the trace of a function in $\BV(\Omega)$ is influenced by how the function is extended outside $\Omega$. 
In this paper, we prove the existence of traces of BV functions with fewer assumptions on the domain than is standard, and
without referring to the behavior of the function in the exterior of the domain. In particular, our 
results hold for domains with various types of cusps, in which either uniformity or Poincar\'e inequalities are violated.

There are four main results in this paper, Theorem~\ref{thm:trace theorem}, Theorem~\ref{thm:L1-est},
Theorem~\ref{thm:density of compactly supported functions revisited-VerI}, and Theorem~\ref{thm:mazya type inequality}.
The first two deal with traces of general $\BV$ functions on a certain class of domains, and the last two deal with 
$\BV$ functions with zero trace.

The structure of the paper is as follows. In Section~2 we give the necessary background definitions used throughout the paper.
In Section~3 we propose a notion of traces of BV functions on an open subset
$\Omega$ of $X$, and in the first main theorem of the paper, Theorem~\ref{thm:trace theorem}, we show that
if the restriction of the measure to $\Omega$  is also doubling and supports a $(1,1)$-Poincar\'e inequality, then
every function in $\BV(\Omega)$ has a trace on  $\partial\Omega$, well-defined up to sets of $\mathcal{H}$-measure zero.
In Section~4 the necessity of the hypotheses imposed on $\Omega$ in
Theorem~\ref{thm:trace theorem} is discussed by means of examples, and in 
Corollary~\ref{cor:exceptional-points} we describe how to relax some of these hypotheses.

The goal of Section~5 is to demonstrate that
under additional geometric assumptions on the boundary of $\Omega$ (regularity of the
$\mathcal{H}$-measure on
$\partial\Omega$), the traces of functions in
$\BV(\Omega)$ lie in $L^1(\partial\Omega,\mathcal{H})$. This result is given in Theorem~\ref{thm:L1-est}. To provide
$L^1$ estimates for the trace, we also prove a BV extension property of domains with such a geometric boundary;
this is Proposition~\ref{thm:extension result}. Section~6 is devoted to the study of $\BV$ functions in $\Omega$ with 
zero trace on $\partial\Omega$. In particular, we show that if  $\mathcal{H}(\partial\Omega)$ is finite for an open set 
$\Omega\subset X$, then every function in $\BV(\Omega)$ with zero trace on $\partial\Omega$ can be extended by
zero to $X\setminus\Omega$ without increasing its BV energy (Theorem~\ref{thm:zero extension}), and that such functions
can be approximated in $\BV(X)$ by BV functions with compact support in $\Omega$ 
(Theorem~\ref{thm:density of compactly supported functions revisited-VerI}). In Section~7 we  prove a Maz'ya-type
Sobolev inequality for functions that vanish on a set of positive capacity (Theorem~\ref{thm:mazya type inequality}).
To prove Theorem~\ref{thm:density of compactly supported functions revisited-VerI} and to study certain locally Lipschitz
approximations of functions in $\BV(\Omega)$, we use Whitney type decompositions of $\Omega$
and discrete convolutions.
For a more substantial description of these
tools, see also the upcoming paper \cite{LaSh}.\\

\noindent {\bf Acknowledgement:} Part of the research for this paper was conducted during the stay of both authors at the
Institut Mittag-Leffler, Sweden, in the fall of 2013, and during the visit of the first author to the University of Cincinnati
in Spring 2015 and the visit of the second author to the Aalto University in June 2015. The authors wish to thank these institutions for their kind hospitality. The research of the second author
is partially supported by the NSF grant DMS\#1200915. The authors also wish to thank Juha Kinnunen and Heikki Hakkarainen
for illuminating discussions on trace results while at the Institut Mittag-Leffler.

\section{Preliminaries}\label{sec:preliminaries}
In this section we introduce the necessary definitions and assumptions.

In this paper, $(X,d,\mu)$ is a complete metric space equipped
with a Borel regular outer measure $\mu$.
The measure is assumed to be doubling, meaning that there exists a constant $C_d>0$ such that
\[
0<\mu(B(x,2r))\leq C_d\mu(B(x,r))<\infty
\]
for every ball $B=B(x,r)$ with center $x\in X$ and radius $r>0$. For a ball $B=B(x,r)$ we will, for brevity, sometimes use 
the notation $\tau B=B(x,\tau r)$, for $\tau>0$. Note that in a metric space, a ball does not necessarily have a unique center 
and radius, but whenever we use the above abbreviation we will consider balls whose center and radii have been specified.

By iterating the doubling condition, we obtain that there are constants $C\ge 1$ and $Q>0$ such that 
\begin{equation}\label{eq:definition of Q}
\frac{\mu(B(y,r))}{\mu(B(x,R))}\ge C^{-1}\, \left(\frac{r}{R}\right)^Q
\end{equation}
for every $0<r\le R$ and $y\in B(x,R)$. See~\cite{BB} for a proof of this.

In general, $C\ge 1$ will denote a constant whose value is not
necessarily the same at each occurrence. When we want to specify that a certain constant 
depends on the parameters $a,b, \ldots,$ we write $C=C(a,b,\ldots)$. Unless otherwise specified, all constants only 
depend on the doubling constant $C_d$ and the constants $C_P,\lambda$ associated
with the Poincar\'e inequality defined later.

We recall that a complete metric space equipped with a doubling measure is proper,
that is, closed and bounded sets are compact. Since $X$ is proper, for any open set $\Omega\subset X$
we define e.g. $\textrm{Lip}_{\loc}(\Omega)$ as the space of
functions that are Lipschitz in every $\Omega'\Subset\Omega$.
Here $\Omega'\Subset\Omega$ means that $\Omega'$ is open and that $\overline{\Omega'}$ is a
compact subset of $\Omega$.

For any set $A\subset X$ and $0<R<\infty$, the restricted spherical Hausdorff content
of codimension $1$ is defined as
\[
\mathcal{H}_{R}(A):=\inf\left\{ \sum_{i=1}^{\infty}
  \frac{\mu(B(x_{i},r_{i}))}{r_{i}}:\,A\subset\bigcup_{i=1}^{\infty}B(x_{i},r_{i}),\,r_{i}\le R\right\}.
\]
The codimension $1$ Hausdorff measure of a set $A\subset X$ is
\begin{equation}\label{eq:codimension 1 Hausdorff measure}
  \mathcal{H}(A):=\lim_{R\rightarrow0}\mathcal{H}_{R}(A).
\end{equation}

The (topological) boundary $\partial E$ of a set $E\subset X$ is
defined as usual. The measure theoretic boundary $\partial^{*}E$ is defined as the set of points $x\in X$
for which both $E$ and its complement have positive upper density, i.e.
\[
\limsup_{r\to 0^+}\frac{\mu(B(x,r)\cap E)}{\mu(B(x,r))}>0\quad\;
  \textrm{and}\quad\;\limsup_{r\to 0^+}\frac{\mu(B(x,r)\setminus E)}{\mu(B(x,r))}>0.
\]

A curve is a rectifiable continuous mapping from a compact interval
to $X$, and is usually denoted by the symbol $\gamma$.
A nonnegative Borel function $g$ on $X$ is an upper gradient 
of an extended real-valued function $u$
on $X$ if for all curves $\gamma$ on $X$, we have
\begin{equation}\label{eq:definition of upper gradients}
|u(x)-u(y)|\le \int_\gamma g\,ds
\end{equation}
whenever both $u(x)$ and $u(y)$ are finite, and 
$\int_\gamma g\, ds=\infty $ otherwise.
Here $x$ and $y$ are the end points of $\gamma$. Upper gradients were originally introduced in \cite{HK}.

If $g$ is a nonnegative $\mu$-measurable function on $X$
and \eqref{eq:definition of upper gradients} holds for $1$-almost every curve,
we say that $g$ is a $1$-weak upper gradient of~$u$. 
By saying that \eqref{eq:definition of upper gradients} holds for $1$-almost every curve
we mean that it fails only for a curve family with zero $1$-modulus. 
A family $\Gamma$ of curves is of zero $1$-modulus if there is a 
nonnegative Borel function $\rho\in L^1(X)$ such that 
for all curves $\gamma\in\Gamma$, the curve integral $\int_\gamma \rho\,ds$ is infinite.

We consider the following norm
\[
\Vert u\Vert_{N^{1,1}(X)}:=\Vert u\Vert_{L^1(X)}+\inf_g\Vert g\Vert_{L^1(X)},
\]
with the infimum taken over all upper gradients $g$ of $u$. 
The Newton-Sobolev space is defined as
\[
N^{1,1}(X):=\{u:\,\|u\|_{N^{1,1}(X)}<\infty\}/{\sim},
\]
where the equivalence relation $\sim$ is given by $u\sim v$ if and only if 
\[
\Vert u-v\Vert_{N^{1,1}(X)}=0.
\]
Similarly, we can define $N^{1,1}(\Omega)$ for any open set $\Omega\subset X$. The space of Newton-Sobolev functions 
with zero boundary values is defined as 
\[
N^{1,1}_0(\Omega):=\{v|_{\Omega}:\,v\in N^{1,1}(X)\textrm{ and }v=0\textrm{ in }X\setminus\Omega\},
\]
i.e. it is the subclass of $N^{1,1}(\Omega)$ consisting of those functions that can be zero extended to the whole space as 
Newton-Sobolev functions. For more on Newton-Sobolev spaces, we refer to \cite{S}, \cite{HKST}, or \cite{BB}.

Next we recall the definition and basic properties of functions
of bounded variation on metric spaces, see \cite{M}. 
For $u\in L^1_{\text{loc}}(X)$, we define the total variation of $u$ as
\[
\|Du\|(X):=\inf\Big\{\liminf_{i\to\infty}\int_X g_{u_i}\,d\mu:\, u_i\in \Lip_{\loc}(X),\, u_i\to u\textrm{ in } L^1_{\text{loc}}(X)\Big\},
\]
where each $g_{u_i}$ is an upper gradient of $u_i$.
We say that a function $u\in L^1(X)$ is \emph{of bounded variation}, 
and denote $u\in\BV(X)$, if $\|Du\|(X)<\infty$. 
A measurable set $E\subset X$ is said to be of \emph{finite perimeter} if $\|D\chi_E\|(X)<\infty$.
By replacing $X$ with an open set $\Omega\subset X$ in the definition of the total variation, we can define $\|Du\|(\Omega)$.
The $\BV$ norm is given by
\begin{equation}\label{eq:def of BV norm}
\Vert u\Vert_{\BV(\Omega)}:=\Vert u\Vert_{L^1(\Omega)}+\Vert Du\Vert(\Omega).
\end{equation}
It was shown in~\cite[Theorem~3.4]{M} that for $u\in\BV(X)$, $\Vert Du\Vert$ is the restriction to the class of open sets of a finite Radon measure defined on the
class of all subsets of $X$. This outer measure is obtained from the map $\Omega\mapsto\Vert Du\Vert(\Omega)$ on open sets
$\Omega\subset X$ via the standard Carath\'eodory construction. Thus, 
for an arbitrary set $A\subset X$, 
\[
\|Du\|(A):=\inf\bigl\{\|Du\|(\Omega):\,\Omega\supset A,\,\Omega\subset X
\text{ is open}\bigr\}.
\]
Similarly, if $u\in\BV(\Omega)$, then $\|Du\|(\cdot)$ is a finite Radon measure on $\Omega$.
We also denote the perimeter of a set $E$ in $\Omega$ by
\[
P(E,\Omega):=\|D\chi_E\|(\Omega).
\]
We have the following coarea formula from~\cite[Proposition 4.2]{M}: if $F\subset X$ is a Borel set and 
$u\in \BV(X)$, then
\begin{equation}\label{eq:coarea}
\|Du\|(F)=\int_{-\infty}^{\infty}P(\{u>t\},F)\,dt.
\end{equation}
If $F$ is open, this holds for every $u\in L^1_{\loc}(F)$.

We will assume throughout the paper that $X$ supports the following $(1,1)$-Poincar\'e inequality:
there are constants $C_P>0$ and $\lambda \ge 1$ such that for every
ball $B(x,r)$, for every locally integrable function $u$,
and for every upper gradient $g$ of $u$, we have 
\[
\vint{B(x,r)}|u-u_{B(x,r)}|\, d\mu 
\le C_P r\vint{B(x,\lambda r)}g\,d\mu,
\]
where 
\[
u_{B(x,r)}:=\vint{B(x,r)}u\,d\mu :=\frac 1{\mu(B(x,r))}\int_{B(x,r)}u\,d\mu.
\]
By approximation, we get the following $(1,1)$-Poincar\'e inequality for $\BV$ functions. There exists $C>0$, depending 
only on the doubling constant and the constants in the Poincar\'e inequality, such that for every ball $B(x,r)$ and every 
$u\in L^1_{\loc}(X)$, we have
\begin{equation}\label{eq:poincare for BV}
\vint{B(x,r)}|u-u_{B(x,r)}|\,d\mu
\le Cr\frac{\Vert Du\Vert (B(x,\lambda r))}{\mu(B(x,\lambda r))}.
\end{equation}

Given an open set $\Omega\subset X$, we can consider it as a metric space in its own right, equipped with the metric 
inherited from $X$ and the restriction of $\mu$ to subsets of $\Omega$. This restriction is a Radon measure on 
$\Omega$, see \cite[Lemma~2.3.15]{HKST}. We say that $\mu$ is doubling on $\Omega$ if the restriction of
$\mu$ to subsets of $\Omega$ is doubling, that is, if there is a constant $C\ge 1$ such that 
\[
0<\mu(B(x,2r)\cap\Omega)\le C\mu(B(x,r)\cap\Omega)<\infty
\]
for every $x\in\Omega$ and $r>0$. Similarly we can require $\Omega$ to support a $(1,1)$-Poincar\'e inequality. 

Given a set $E\subset X$ of locally finite perimeter, for $\mathcal H$-a.e. $x\in \partial^*E$ we have
\begin{equation}\label{eq:density of E}
\gamma \le \liminf_{r\to 0} \frac{\mu(E\cap B(x,r))}{\mu(B(x,r))}
   \le \limsup_{r\to 0} \frac{\mu(E\cap B(x,r))}{\mu(B(x,r))}\le 1-\gamma
\end{equation}
where $\gamma \in (0,1/2]$ only depends on the doubling constant and the constants in the Poincar\'e inequality, 
see~\cite[Theorem 5.4]{A2}. 
We denote the collection of all points $x$ that satisfy~\eqref{eq:density of E} by $E_\gamma$.
For a Borel set $F\subset X$ and a set $E\subset X$ of finite perimeter, we know that
\begin{equation}\label{eq:def of theta}
\Vert D\chi_{E}\Vert(F)=\int_{\partial^{*}E\cap F}\theta_E\,d\mathcal H,
\end{equation}
where $\partial^*E$ is the measure-theoretic boundary of $E$ and
$\theta_E:X\to [\alpha,C_d]$, with $\alpha=\alpha(C_d,C_P,\lambda)>0$, see \cite[Theorem 5.3]{A2} 
and \cite[Theorem 4.6]{AMP}.

In the metric setting it is not known, in general, whether the condition $\mathcal H(\partial^*E)<\infty$ for a $\mu$-measurable 
set $E\subset X$ implies that $P(E,X)<\infty$. We say that $X$ supports a \emph{strong relative isoperimetric inequality} 
if this is true, see \cite{KKST} or \cite{KLS} for more on this question.

The jump set of $u\in\BV(X)$ is defined as
\[
S_{u}:=\{x\in X:\, u^{\wedge}(x)<u^{\vee}(x)\},
\]
where $u^{\wedge}(x)$ and $u^{\vee}(x)$ are the lower and upper approximate limits of $u$ defined by
\begin{equation}\label{eq:lower approximate limit}
u^{\wedge}(x):
=\sup\left\{t\in\overline\R:\,\lim_{r\to 0^+}\frac{\mu(B(x,r)\cap\{u<t\})}{\mu(B(x,r))}=0\right\}
\end{equation}
and
\begin{equation}\label{eq:upper approximate limit}
u^{\vee}(x):
=\inf\left\{t\in\overline\R:\,\lim_{r\to 0^+}\frac{\mu(B(x,r)\cap\{u>t\})}{\mu(B(x,r))}=0\right\}.
\end{equation}
We also set 
\begin{equation}\label{eq:pointwise-rep}
\widetilde{u}:=\frac{u^{\wedge}+u^{\vee}}{2}.
\end{equation}

By \cite[Theorem 5.3]{AMP}, the variation measure of a $\BV$ function can be decomposed into the absolutely 
continuous and singular part, and the latter into the Cantor and jump part, as follows. Given an open set 
$\Omega\subset X$ and $u\in\BV(\Omega)$, we have
\begin{equation}\label{eq:decomposition}
\begin{split}
\Vert Du\Vert(\Omega)
&=\Vert Du\Vert^a(\Omega)+\Vert Du\Vert^s(\Omega)\\
&=\Vert Du\Vert^a(\Omega)+\Vert Du\Vert^c(\Omega)+\Vert Du\Vert^j(\Omega)\\
&=\int_{\Omega}a\,d\mu+\Vert Du\Vert^c(\Omega)
   +\int_{\Omega\cap S_u}\int_{u^{\wedge}(x)}^{u^{\vee}(x)}\theta_{\{u>t\}}(x)\,dt\,d\mathcal H(x)
\end{split}
\end{equation}
where $a\in L^1(\Omega)$ is the density of the absolutely continuous part and the functions $\theta_{\{u>t\}}$ 
are as in~\eqref{eq:def of theta}. From this decomposition it also follows that 
$\mathcal{H}$ is a $\sigma$-finite measure on $S_u$.

\section{Traces of $\BV$ functions}

We give the following definition for the boundary trace, or trace for short, of a function defined on an open set.
\begin{definition}
Let $\Omega\subset X$ be an open set and let $u$ be a $\mu$-measurable function on $\Omega$. 
A function $Tu:\partial\Omega\to\R$
is the trace of $u$ if for $\mathcal{H}$-almost every $x\in\partial\Omega$ we have  
\begin{equation}\label{eq:definition of traces}
\lim_{r\to 0^+}\,\vint{\Omega\cap B(x,r)}|u-Tu(x)|\,d\mu= 0.
\end{equation}
\end{definition}

\begin{example}\label{ex:slit disk}
Consider $X=\C=\R^2$, and set
\[
\Omega=B(0,1)\setminus \{z=(x_1,x_2):\,x_1>0,\,x_2=0\},
\]
that is, the slit disk, and let $u(z):=\Arg(z)$. The function $u$ does not have a trace at any 
boundary point in the slit, but we do have
\[
\lim_{r\to 0^+}\,\vint{\Omega\cap B(z,r)}u\,d\mathcal L^2= \pi
\]
for every $z=(x_1,0)$ with $0<x_1<1$. For this reason, it is crucial that we define traces by requiring the stronger 
condition~\eqref{eq:definition of traces}.
\end{example}

We start by showing that for sufficiently regular domains, BV functions can be extended from the domain to its 
closure, which we consider as a metric space in its own right.
In the following, we define $\overline{\mu}$ as the zero extension of $\mu$ from $\Omega$ to $\overline{\Omega}$,
that is, for $A\subset\overline{\Omega}$ we set $\overline{\mu}(A)=\mu(A\cap\Omega)$. 
By~\cite[Lemma 2.3.20]{HKST} we know that $\overline{\mu}$ is a Borel regular outer measure on 
$\overline{\Omega}$. Also, the zero extension of any $\mu$-measurable function on $\Omega$ to 
either $\overline{\Omega}$ or the whole space $X$ is $\mu$-measurable, see \cite[Lemma 2.3.22]{HKST}.

\begin{proposition}\label{prop-vanishingBVonBoundary}
Assume that $\Omega$ is a bounded open set that supports a $(1,1)$-Poincar\'e inequality, and $\mu$ is doubling on $\Omega$. 
Equip the closure $\overline{\Omega}$ with $\overline{\mu}$.
If $u\in \BV(\Omega)$, then the zero 
extension of $u$ to $\overline{\Omega}$, denoted by $\overline{u}$, satisfies 
$\Vert \overline{u}\Vert_{\BV(\overline{\Omega})}=\Vert u\Vert_{\BV(\Omega)}$ and thus 
$\Vert D\overline{u}\Vert(\partial\Omega)=0$.
\end{proposition}

\begin{proof}
Clearly $\Vert u\Vert_{L^1(\Omega)}=\Vert \overline{u}\Vert_{L^1(\overline{\Omega})}$. Since 
$\Omega$ is a metric space with a doubling measure supporting a $(1,1)$-Poincar\'e inequality, 
we know that Lipschitz functions are dense in $N^{1,1}(\Omega)$, see \cite[Theorem 5.1]{BB}. 
Thus for $u\in\BV(\Omega)$ we have a sequence of Lipschitz functions $u_i$ on $\Omega$ 
with $u_i\to u$ in $L_{\loc}^1(\Omega)$ and
\[
\Vert Du\Vert(\Omega)=\liminf_{i\to \infty}\int_{\Omega}g_{u_i}\,d\mu,
\]
where each $g_{u_i}$ is an upper gradient of $u_i$.
By using the fact that $\Omega$ is bounded and by considering truncations of $u$, we can assume that in fact 
$u_i\to u$ in $L^1(\Omega)$. We can extend every Lipschitz function $u_i$ on $\Omega$ to a Lipschitz function on 
$\overline{\Omega}$, still denoted by $u_i$.
Then $u_i\to \overline{u}$ in $L^1(\overline{\Omega})$, and
\[
\Vert Du\Vert(\Omega)=\liminf_{i\to \infty}\int_{\Omega}g_{u_i}\,d\mu
   =\liminf_{i\to \infty}\int_{\overline{\Omega}}g_{u_i}\,d\overline{\mu}\ge \Vert D\overline{u}\Vert(\overline{\Omega}).
\]
The last inequality follows from the fact that the zero extension of $g_{u_i}$ to $\partial\Omega$ is a $1$-weak upper gradient of $u_i$ in $\overline{\Omega}$, by \cite[Lemma 5.11]{BS}.
Thus we have $\overline{u}\in \BV(\cOm)$ with $\Vert D\overline{u}\Vert(\overline{\Omega})=\Vert Du\Vert(\Omega)$, 
so it follows that $\Vert D\overline{u}\Vert(\bdy)=0$.
\end{proof}

Recall the definition of the codimension 1 Hausdorff measure $\mathcal H$ from \eqref{eq:codimension 1 Hausdorff measure}. 
We denote by $\overline{\mathcal H}$ the codimension 1 Hausdorff measure in the space $\overline{\Omega}$, defined with 
respect to the measure $\overline{\mu}$.

We say that an open set $\Omega$ satisfies a \emph{measure density condition} if there is a constant $c_m>0$ such that
\begin{equation}\label{eq:measure density condition}
\mu(B(x,r)\cap\Omega)\ge c_m\mu(B(x,r))
\end{equation}
for $\mathcal H$-a.e. $x\in\partial\Omega$ and every $r\in (0,\diam(\Omega))$.

\begin{theorem}\label{thm:trace theorem}
Let $\Omega$ be a bounded open set that supports a $(1,1)$-Poincar\'e inequality, and $\mu$ be doubling in 
$\Omega$. Then there is a linear trace operator $T$ on $\BV(\Om)$ such that given $u\in \BV(\Om)$, for 
$\overline{\mathcal H}$-a.e. $x\in\partial\Omega$ we have 
\[
\lim_{r\to 0^+}\,\vint{B(x,r)\cap\Om}|u-Tu(x)|^{Q/(Q-1)}\, d\mu=0.
\]
If $\Omega$ also satisfies the measure density condition \eqref{eq:measure density condition}, the above 
holds for $\mathcal H$-a.e. $x\in\partial\Omega$. 
\end{theorem}

\begin{proof}
By \cite[Lemma 7.2.3]{HKST} we know that $\overline{\Omega}$ equipped with $\overline{\mu}$ also supports a 
$(1,1)$-Poincar\'e inequality, and $\overline{\mu}$ is doubling on $\overline{\Omega}$. By 
Proposition~\ref{prop-vanishingBVonBoundary} we know that
$\Vert D\overline{u}\Vert(\partial\Omega)=0$ and so by the decomposition~\eqref{eq:decomposition}, we also
know that $\overline{\mathcal H}(S_{\overline{u}}\cap\bdy)=0$. On the other hand, by the Lebesgue point 
theorem given in~\cite[Theorem~3.5]{KKST2}, we
have for $\overline{\mathcal H}$-a.e. $x\in\bdy\setminus S_{\overline{u}}$
\[
   \lim_{r\to 0^+}\vint{B(x,r)\cap\Om}|u-\overline{u}^\wedge(x)|^{Q/(Q-1)}\, d\mu
       = \lim_{r\to 0^+}\vint{B(x,r)}|\overline{u}-\overline{u}^\wedge(x)|^{Q/(Q-1)}\, d\overline{\mu}=0.
\]
Thus we set $Tu(x)=\overline{u}^\wedge(x)$. Note that because $\overline{\mu}(\partial\Omega)=0$,
the value of $\overline{u}^\wedge$ at points in $\partial\Omega$ is unaffected by how we extend $u$ to
$\partial\Omega$. 
The linearity of $T$ is immediate. If the measure density condition~\eqref{eq:measure density condition} holds, then
$\mathcal H$ and  $\overline{\mathcal H}$ are comparable for subsets of $\partial\Omega$, and so the result holds 
for $\mathcal H$-a.e. $x\in\partial\Omega$. 
\end{proof}

\section{Some examples}

In this section we consider examples that illustrate the function of each of the hypotheses stated in
Theorem~\ref{thm:trace theorem}.

Theorem~\ref{thm:trace theorem} is stronger than it seems, 
for we do \emph{not} assume that $\mathcal{H}(\partial\Omega)$
is finite. 

\begin{example} 
The von Koch snowflake domain in the plane is a uniform domain, and uniform domains 
equipped with the restriction of the Lebesgue measure have a doubling measure supporting a $(1,1)$-Poincar\'e 
inequality, see e.g. \cite{BS}. On the other hand, the $\mathcal{H}$-measure (that is, the 
one-dimensional Hausdorff measure) of the boundary of the snowflake
domain is infinite, so the existence of $Tu$ at $\mathcal{H}$-almost every boundary point is a strong statement.
However, we do not claim here that the trace $Tu$ is in the class $L^1(\partial\Omega,\mathcal{H})$, 
and indeed this would be too much to hope for.
Since constant functions are in the class $\BV(\Omega)$ whenever
$\Omega$ is a bounded domain, and their traces are also constant functions on $\partial\Omega$, in order to consider
whether traces of $\BV(\Omega)$-functions are in $L^1(\partial\Omega,\mathcal{H})$ it is necessary that 
$\mathcal{H}(\partial\Omega)$ be finite. Observe also that for the snowflake domain,
$\partial^*\Omega=\partial\Omega$ so considering $L^1(\partial^*\Omega,\mathcal{H})$ does not help either. 
\end{example}

We will consider $L^1$ estimates in Section~\ref{L1-estimates}.

\begin{example}
Let $C\subset [0,1]$ 
be the usual $1/3$-Cantor set, and define $\Omega:=(0,1)^2\setminus C\times C$. This is an open set from which 
Sobolev functions can be extended to the whole unit square, as can be seen by the characterization of Sobolev 
functions by means of absolute continuity on almost every line parallel to the (canonical) coordinate axes of 
Euclidean spaces. Thus $\Omega$ supports a $(1,1)$-Poincar\'e inequality, and it is 
clear that the Lebesgue measure is doubling in $\Omega$ and that the measure density 
condition~\eqref{eq:measure density condition} holds. Therefore by Theorem~\ref{thm:trace theorem}, any function 
$u\in\BV(\Omega)$ has a trace at $\mathcal H$-a.e. $x\in\partial\Omega$. However, the Hausdorff dimension of
$C\times C$ is $2 \log(2)/\log(3)>1$, and so $\mathcal H(\partial\Omega)=\infty$. Together with the example discussed 
above, this indicates that in the generality considered in Theorem~\ref{thm:trace theorem}
we do not have $Tu\in L^1(\partial\Omega, \mathcal{H})$.
\end{example}

To see why the assumptions of Theorem~\ref{thm:trace theorem} are necessary at least in some form, we first 
note that without a Poincar\'e 
inequality, $\Omega$ can be chosen to be the slit disk from Example~\ref{ex:slit disk}, where we know that
traces of functions of bounded variation do not exist. On the other hand, consider 
a domain in $\R^2$ with an interior cusp:
\[
 \Omega=\{x=(x_1,x_2)\in\R^2:\, |x|<1\text{ and }|x_2|>x_1^2\text{ when }x_1\ge 0\},
\]
This does not support a $(1,1)$-Poincar\'e inequality, since such an inequality always implies that the space 
(in this case the set $\Omega$) is quasiconvex, that is, every pair of points can be connected by a curve whose 
length is at most a constant times the distance between the two points 
(see e.g.~\cite[Proposition~4.4]{HaKo}). However, boundary traces of $\BV$ functions 
do exist for this domain, as can be deduced from the following result. In the above domain with an
internal cusp, we can take 
the sets $F_{\eps}$ to be small closed balls centered at the origin.

\begin{corollary}\label{cor:exceptional-points}
Let $\Omega$ be an open set. The conclusions of Theorem \ref{thm:trace theorem} are true if for every $\eps>0$ there is a 
closed set $F_{\eps}\subset X$ with
\begin{equation}\label{eq:removing a closed set}
\overline{\cH}(\partial\Omega\cap F_{\eps})<\eps
\end{equation}
such that
$\Om\setminus F_{\eps}$, instead of $\Omega$ itself, satisfies the hypotheses of 
Theorem~\ref{thm:trace theorem} (apart from the last sentence).
\end{corollary}


\begin{proof}
Fix $\eps>0$. It is easy to check that $\partial\Omega\setminus F_{\eps}\subset \partial(\Omega\setminus F_{\eps})$. 
As earlier, we can define a codimension 1 Hausdorff measure in the closure of $\Omega\setminus F_{\eps}$ with 
respect to the zero extension of $\mu$ from $\Omega\setminus F_{\eps}$, but since $F_{\eps}$ is closed, 
this agrees with $\overline{\mathcal H}$ on $\partial\Omega\setminus F_{\eps}$. Thus by 
Theorem~\ref{thm:trace theorem} we know that for $\overline{\mathcal H}$-a.e. 
$x\in \partial\Omega\setminus F_{\eps}$, there exists $Tu(x)\in\R$ with
\[
\lim_{r\to 0^+}\,\vint{B(x,r)\cap\Om\setminus F_{\eps}}|u-Tu(x)|^{Q/(Q-1)}\, d\mu=0.
\]
But since $F_{\eps}$ is closed, this immediately implies that
\[
\lim_{r\to 0^+}\,\vint{B(x,r)\cap\Om}|u-Tu(x)|^{Q/(Q-1)}\, d\mu=0.
\]
Thus the trace $Tu(x)$ exists outside a subset of $\partial\Omega$ with $\overline{\mathcal H}$-measure at most 
$\eps$, due to \eqref{eq:removing a closed set}. By letting $\eps\to 0$, we get the result.
\end{proof}

In the following example, we consider the distinction between the measures $\mathcal H$ and $\overline{\mathcal H}$ .

\begin{example}\label{ex:exterior cusp counterexample}
Without the measure density condition~\eqref{eq:measure density condition}, the conclusion of 
Theorem~\ref{thm:trace theorem} does not 
necessarily hold for $\mathcal H$-a.e. $x\in\partial\Omega$. As a counterexample, consider the space $X=\R^2$ 
equipped with the Euclidean metric and the weighted Lebesgue measure $d\mu:=w\,d\mathcal{L}^2$ with $w=|x|^{-1}$. 
It can be shown that $\mu$ is doubling and $X$ supports a $(1,1)$-Poincar\'e inequality, 
see~\cite[Corollary 15.34]{HKM} and~\cite[Theorem 4]{JB}. Consider the domain
\begin{equation}\label{eq:definition of exterior cusp}
\Omega:=\left\{x=(x_1,x_2)\in \R^2 :\,0<x_1<1,\,|x_2|< x_1^2\right\}.
\end{equation}
We can check that $\mu$ is doubling on $\Omega$ and that $\Omega$ supports a $(1,1)$-Poincar\'e inequality, see the 
upcoming Example~\ref{ex:cusp supports Poincare inequality}. On the other hand, the measure density 
condition~\eqref{eq:measure density condition} clearly 
does not hold at the origin $(0,0)$, for $\mu(B((0,0),r))= 2\pi r$, whereas $\mu(B((0,0),r)\cap\Omega)\le r^2$. 
For $(0,0)\in\partial\Omega$, if $B$ is a ball
containing $(0,0)$,
then $B((0,0),2r)$ contains this ball, where $r$ is the radius of $B$. Thus by the doubling property of $\mu$, we have
\[
  \frac{\mu(B)}{r}\le 2\frac{\mu(B((0,0),2r))}{2r}\le C \frac{\mu(B)}{r}.
\]
Therefore
\[
\mathcal H(\{(0,0)\})\ge \frac{1}{C}\lim_{r\to 0^+}\frac{\mu(B((0,0),r))}{r}=\frac{1}{C}\lim_{r\to 0^+}\frac{2\pi r}{r}=\frac{2\pi}{C},
\]
whereas
\[
\overline{\mathcal H}(\{(0,0)\})\le\lim_{r\to 0^+}\frac{\mu(B((0,0),r)\cap\Omega)}{r} \le\lim_{r\to 0^+}\frac{r^2}{r}=0.
\]
Now, if we define the function $u:=|x|^{-1/2}$, and denote its local Lipschitz constant by $\Lip u$, we have
\[
\Vert Du\Vert(\Omega)=\int_{\Omega}\Lip u\,d\mu= \frac 12 \int_{\Omega} |x|^{-3/2}\,d\mu\le \int_0^1 r^{-3/2} r^2 r^{-1}\,dr<\infty,
\]
and similarly $\Vert u\Vert_{L^1(\Omega)}<\infty$, so $u\in\BV(\Omega)$. However,  the trace of $u$ does not exist 
at $(0,0)$. So without the measure density condition, we may have traces for $\overline{\mathcal H}$-a.e. but not 
$\mathcal H$-a.e. $x\in\partial\Omega$. Thus by allowing ourselves the flexibility of working with 
either $\mathcal{H}$ or $\overline{\mathcal{H}}$, and by taking Corollary~\ref{cor:exceptional-points} into 
account, we cover a wider class of domains.
\end{example}

Next we show that various exterior cusps do satisfy the assumptions of Theorem \ref{thm:trace theorem}, apart from 
the measure density condition~\eqref{eq:measure density condition}.

\begin{example}\label{ex:cusp supports Poincare inequality}
Take the unweighted space $X=\R^2$, and as in \eqref{eq:definition of exterior cusp}, define the domain with a cusp
\[
\Omega:=\left\{x=(x_1,x_2)\in \R^2\,:\,0<x_1<1,\,|x_2|< x_1^2\right\}.
\] 
To show that $\mu$ is doubling on $\Omega$, it suffices to show the doubling property with respect to
squares centered at points in $\Omega$. For $x=(x_1, x_2)\in\Omega$ and $r>0$, this can be done 
by an explicit computation for two separate cases,
with either $r\ge x_1^2$ or $r<x_1^2$ --- we leave the details to the reader.

Now we focus on showing that 
$\Omega$ supports a $(1,1)$-Poincar\'e inequality. For this, it is enough to show that $\Omega$ supports a Semmes 
family of curves, see \cite{Sem} or \cite{Hei01}. Pick a pair of points $x,y\in\Omega$; here we only consider the case 
$x=(x_1,0)$ and $y=(y_1,0)$. The required condition is that there is a family $\Gamma_{x,y}$ of curves connecting 
$x$ and $y$, and a probability measure $\alpha_{x,y}$ on this family, such that
\begin{equation}\label{eq:semmes definition}
\begin{split}
&\int_{\Gamma_{x,y}}\int_{\gamma}\chi_A(t)\,dt\,d\alpha_{x,y}(\gamma)\\
&\qquad \le  C\int_{A} \frac{d(z,x)}{\mathcal L^2|_{\Omega}(B(x,d(z,x)))}
    +\frac{d(z,y)}{\mathcal L^2|_{\Omega}(B(y,d(z,y)))}\,d\mathcal L^2|_{\Omega}(z)
\end{split}
\end{equation}
for every Borel set $A\subset \Omega$.
Define the Semmes family $\Gamma_{x,y}$ of curves between $x$ and $y$ as follows: for $s\in (-1,1)$,
\begin{equation}\label{eq:defining the semmes family}
\begin{split}
&\gamma_s(t):=
\begin{cases}
(x_1+t,st)\\
(x_1+t,s(x_1+t)^2)\qquad \\
(x_1+t,-s(t+x_1-y_1))
\end{cases}
\end{split}
\end{equation}
for the respective cases 
\begin{align*}
&\begin{cases}
t\ge 0,&\quad y_1-(x_1+t)\ge (x_1+t)^2, \quad (x_1+t,t)\in \overline{\Omega},\\
t\ge 0,&\quad y_1-(x_1+t)\ge (x_1+t)^2, \quad (x_1+t,t)\notin \overline{\Omega},\\
0\le t\le y_1-x_1,&\quad y_1-(x_1+t)< (x_1+t)^2,
\end{cases}
\end{align*}
respectively. Note that the parametrization is not by arc-length, but comparable to it. Essentially, the curves spread 
out from $x$ as a "pencil" of angle $\pi/2$ as long as there is space in $\Omega$, then they become parabolas that 
correspond to the cusp, and finally the curves converge on y as another "pencil" of angle $\pi/2$.

Since the curves $\gamma_s\in\Gamma_{x,y}$ are parametrized by $s$, we can define $\alpha_{x,y}:=\frac 12 \mathcal H^1|_{(-1,1)}$, where $\mathcal H^1$ is the 1-dimensional Hausdorff measure.
Now, if $A\subset \Omega$ is such that every $\gamma_s$ is defined in $A$ according to the first or 
third possible definition given in \eqref{eq:defining the semmes family}, it is a standard result, possible to 
prove by the classical  coarea formula, that \eqref{eq:semmes definition} holds.

On the other hand, if $A\subset \Omega$ is such that every $\gamma_s$ is defined by the second definition in $A$, 
 for every $z\in A$, we have
\begin{equation}\label{eq:Riesz kernel estimate}
\begin{split}
\frac{d(z,x)}{\mathcal L^2(B(x,d(z,x))\cap\Omega)}
&\ge \frac 12 \frac{z_1-x_1}{\int_{x_1}^{x_1+2(z_1-x_1)}2 v^2\,dv}\\
&=   \frac{1}{8} \left(\,\vint{\{(x_1,x_1+2(z_1-x_1))\}} v^2\,dv\right)^{-1}\\
&\ge \frac 18 \frac{1}{z_1^2},
\end{split}
\end{equation}
where the last inequality follows from the fact that $v\mapsto v^2$ is convex.
On the other hand, by the classical coarea formula we have
\[
\int_{-1}^1\int_{\gamma_s}\chi_A(t)\,dt\, d\mathcal H^1(s)\le \int_A \frac{1}{z_1^2}\,d\mathcal L^2.
\]
Thus condition \eqref{eq:semmes definition} holds in this case as well.

Finally, it can be checked that essentially the same definition of the Semmes family works in the more 
general case of the space $\R^n$, $n\in\N$, $n\ge 2$, with a weighted Lebesgue measure 
$d\mu=w\, d\mathcal{L}^n$ with $w=|x|^{\alpha}$, $\alpha\in\R$, and with a polynomial cusp of degree 
$\beta>0$ rather than 2. The above computations run through with small changes at least if we require that
$\alpha+\beta\ge 1$.
\end{example}

\section{$L^1$ estimates for traces}\label{L1-estimates}

In this section we consider a particular condition on the domain $\Omega$ (given below 
in~\eqref{eq:codimension boundary condition}) which, in addition to the assumptions on $\Omega$ given in Section~3, 
is necessary and sufficient for obtaining
$L^1$ estimates for traces. Such estimates are related to extensions of $\BV$ functions ---
recall that a domain $\Omega$ is a BV extension domain if there is a bounded operator
$E:\,\BV(\Omega)\to \BV(X)$
such that $Eu\vert_\Omega=u$.
By \cite[Proposition 6.3]{KKST} we know that if $\Omega$ is an open set with $\mathcal H(\partial\Omega)<\infty$ 
and $E\subset\Omega$ is a $\mu$-measurable set with $P(E,\Omega)<\infty$, then $P(E,X)<\infty$. However, 
here we do not have an extension bound $P(E,X)\le CP(E,\Omega)$.
It was shown by Burago and Maz'ya in~\cite{BuMa} that
a domain $\Omega$ is a weak $\BV$ extension domain (with control only on the total variation of the extension, not the whole $\BV$ norm) if and only if there is a constant $C\ge 1$ 
such that whenever $E\subset\Omega$
is of finite perimeter in $\Omega$ and of sufficiently small diameter, there is a set $F\subset X$ with $F\cap\Omega=E$ such that 
$P(F,X)\le CP(E,\Omega)$. See~\cite[Theorem 3.8]{BM} for a proof of this fact in the metric setting. 
In the following extension result we get a  bound not only on the total variation but on the whole $\BV$ norm.
Some of the assumptions we require are 
somewhat simpler than in~\cite{BS}, where Newtonian functions are extended. 

\begin{remark}\label{rem1}
While we are mostly interested in applying extension results to estimates for traces,
the trace operator $T$ need not exist for a BV extension domain. This is demonstrated by the planar 
slit disk, which is known to be a $\BV$ extension domain by~\cite[Theorem 1.1]{KMS}. Furthermore, 
we can show that every Sobolev $N^{1,1}$-extension domain is a BV extension domain, but not every
BV extension domain is a Sobolev $N^{1,1}$-extension domain, 
as demonstrated again by the slit disk.
\end{remark}

\begin{proposition}\label{thm:extension result}
Assume that $\Omega\subset X$ is a bounded domain that
supports a $(1,1)$-Poincar\'e inequality, that there is a constant $C_{\partial\Omega}>0$ such that
\begin{equation}\label{eq:codimension boundary condition}
\mathcal H(B(x,r)\cap \partial\Omega)\le C_{\partial\Omega}\frac{\mu(B(x,r))}{r}
\end{equation}
for all $r\in(0,2\diam(\Omega))$,
and that there is a constant $c_m>0$ such that the measure density condition~\eqref{eq:measure density condition}
holds for all $x\in\partial\Omega$ and all $r\in(0,\diam(\Omega))$. 
Then there is a constant 
$C_{\Omega}=C_{\Omega}(C_d,C_P,\lambda,C_{\partial\Omega},c_m,\diam(\Omega))$ such that for every 
$u\in\BV(\Omega)$ the zero extension of $u$, denoted by $\widehat{u}$, satisfies
\begin{equation}\label{eq:extension bound}
\Vert \widehat{u}\Vert_{\BV(X)}\le C_{\Omega}\Vert u\Vert_{\BV(\Omega)}.
\end{equation}
\end{proposition}

\begin{example}
Condition~\eqref{eq:codimension boundary condition} is not needed for $\Omega$ to be  a
$\BV$ extension domain, as demonstrated by the von Koch snowflake domain. However, this 
domain does not satisfy~\eqref{eq:extension bound}. While the boundary of the von Koch snowflake domain has
infinite $1$-dimensional Hausdorff measure, a modification of it will result in a domain
which satisfies $\mathcal H(\partial\Omega)<\infty$ as well as all
the hypotheses of the above theorem except for~\eqref{eq:codimension boundary condition}, and it can be seen
directly that in such a domain, the zero extensions of BV functions are of finite perimeter in 
$X=\R^2$ but do not satisfy~\eqref{eq:extension bound}.
One such domain can be obtained as follows. The curve obtained by the construction of the von Koch snowflake curve
at the $k$-th step has length $(4/3)^k$. For each positive integer $k$ we scale a copy of the $k$-th step of the snowflake
curve by a factor $(3/4)^{2k}$, and concatenate them so that their end points lie on the $x$-axis in $\R^2$. 
Such a curve has total length $\sum_{k=1}^\infty (3/4)^k<\infty$. We replace one side of a square of length
$\sum_{k=1}^\infty (3/4)^{2k}$ with this curve. The interior region in $\R^2$ surrounded by this closed curve
is a uniform domain, and hence the restriction of the Lebesgue measure to this domain is doubling and supports a
$(1,1)$-Poincar\'e inequality. Furthermore, the $\mathcal{H}$-measure of its boundary is
\[
\sum_{k=1}^\infty (3/4)^k+ 3\, \sum_{k=1}^\infty (3/4)^{2k}<\infty.
\]
However, this domain fails to satisfy~\eqref{eq:codimension boundary condition}.

To demonstrate that we cannot discard the other assumptions of Proposition~\ref{thm:extension result} 
either, we note that without the assumption of a Poincar\'e inequality in 
$\Omega$, we could take the space to be
\[
X:=\{(x_1,x_2)\in\R^2:\,0\le x_1\le 1,\,-1\le x_2\le x_1^2\}
\]
and then set $\Omega$ to be the domain
\[
\Omega:=X\setminus [0,1/2]\times\{0\}.
\]
Note that $\Omega$ satisfies the measure density condition and~\eqref{eq:codimension boundary condition}, 
but fails to support a $(1,1)$-Poincar\'e inequality.
If we now consider the sets $E_i:=\{(x_1,x_2)\in X :\, 0\le x_1\le 1/i,\, x_2>0\}$, $i\in\N$, then
$P(E_i,\Omega)=1/i^2$, whereas for every $F_i\subset X$ for which $F_i\cap\Omega=E_i$,
we must have $P(F_i,X)\ge 1/i$. Therefore we cannot find a constant $C\ge 1$ such that
$P(F_i,X)\le C P(E_i,\Omega)$ for all positive integers $i$.

To see that the measure density condition is also needed in the proposition, we consider $X$ to be the Euclidean plane, with
$\Omega=\{(x_1,x_2)\in\R^2 :\, 0<x_1<1,\, |x_2|<x_1^2\}$ the external cusp domain we saw in 
Example~\ref{ex:cusp supports Poincare inequality}. As shown in that example, $\Omega$ supports a 
$(1,1)$-Poincar\'e inequality
and the restriction of the Lebesgue measure to $\Omega$ is doubling, and 
clearly~\eqref{eq:codimension boundary condition} holds as well. However, $\Omega$ does not satisfy the measure
density condition, and the sets $E_i:=\{(x_1,x_2)\in\R^2 :\, 0<x_1<1/i,\, |x_2|< x_1^2\}$,
$i\in\N$, together demonstrate that we cannot
find a bound $C\ge 1$ controlling the perimeter of the extensions.
\end{example}


\begin{proof}[Proof of Proposition~\ref{thm:extension result}]
We first consider a $\mu$-measurable set $E$ such that $P(E,\Omega)<\infty$.
By~\eqref{eq:codimension boundary condition} we have $\mathcal H(\partial\Omega)<\infty$. Therefore
by~\cite[Proposition 6.3]{KKST} we know that $P(E,X)<\infty$.
We consider the following two cases.

\noindent {\bf Case 1:} $\mu(E)>\frac 12 \mu(\Omega)$. Then by \eqref{eq:def of theta},
\begin{align*}
 P(E,X)
 &\le P(E,\Omega)+P(E,\partial\Omega)\\
 &\le P(E,\Omega)+C \mathcal{H}(\partial\Omega)\\
 &\le P(E,\Omega)
     +2C\frac{\mathcal{H}(\partial\Omega)}{\mu(\Omega)}\mu(E)\le C \Vert\chi_E\Vert_{\BV(\Omega)}
\end{align*}
for some $C=C(C_d,C_P,\lambda,C_{\partial\Omega},c_m,\diam(\Omega))$, as desired.
%

\noindent {\bf Case 2:}
$\mu(E)\le \frac 12 \mu(\Omega)$. The perimeter measure of $E$ is carried on the set 
$E_\gamma$, where 
$E_\gamma$ consists of all the points $x\in\partial E$ that satisfy~\eqref{eq:density of E}, see~\cite[Theorem~5.4]{A2}.
We therefore need to control $\mathcal{H}(E_\gamma\cap\partial\Omega)$ in terms of $P(E,\Omega)$.
Fix $x\in E_\gamma\cap\partial\Omega$.
Since $x\in E_\gamma$, for some $r_0>0$ we have that for all $0<r<r_0$,
\[
  \frac{\mu(B(x,r)\cap E)}{\mu(B(x,r)\cap\Omega)}>\frac{\gamma}{2}.
\]
Fix such $r>0$. Because $\mu(E)\le \frac 12 \mu(\Omega)$, we can choose the smallest $j\in\N$ such that 
\[
  \frac{\mu(B(x,2^jr)\cap E)}{\mu(B(x,2^jr)\cap\Omega)}\le \frac{1}{2}.
\] 
If $j=1$, then by the choice of $r$ we have
\[
 \frac{\mu(B(x,r)\cap E)}{\mu(B(x,r)\cap\Omega)}>\frac{\gamma}{2}.
\]
If $j>1$, then we again have
\[
   \frac{\mu(B(x,2^{j-1}r)\cap E)}{\mu(B(x,2^{j-1}r)\cap\Omega)}>\frac{1}{2}\ge \frac{\gamma}{2},
\]
and so by the doubling property of $\mu$ together with the measure density condition for $\Omega$,
\[
 \frac{1}{2}\ge \frac{\mu(B(x,2^jr)\cap E)}{\mu(B(x,2^jr)\cap\Omega)}
      \ge \frac{\mu(B(x,2^{j-1}r)\cap E)}{\mu(B(x,2^{j-1}r)\cap\Omega)}\frac{\mu(B(x,2^{j-1}r)\cap\Omega)}{\mu(B(x,2^jr)\cap\Omega)}
      \ge \frac{\gamma}{2} \frac{c_m}{C_d}.
\]
Set $r_x:=2^j r$.
By using the relative isoperimetric inequality in $\Omega$, the measure density condition, and the above, we get
\begin{align*}
C P(E,\Omega\cap B(x, \lambda r_x))
\ge \frac{\gamma}{2} \frac{c_m}{C_d}\frac{\mu(B(x,r_x)\cap \Omega)}{r_x}
\ge \frac{\gamma}{2} \frac{c_m^2}{C_d^2}\frac{\mu(B(x,r_x))}{r_x}.
\end{align*}
Thus we get a covering
$\{B(x,\lambda r_x)\}_{x\in \partial\Omega\cap E_{\gamma}}$ of $\partial\Omega\cap E_\gamma$
in which every ball $B(x,r_x)$ satisfies the above inequality.
By the $5$-covering theorem, we can pick a countable collection $\{B_j=B(x_j, \lambda r_j)\}_{j=1}^{\infty}$ of pairwise 
disjoint balls such that $\{B(x_j,5\lambda  r_j)\}_{j=1}^{\infty}$ is a cover of $\partial\Omega\cap E_{\gamma}$. 
Therefore by~\eqref{eq:codimension boundary condition} and the doubling property of $\mu$, together with the fact
that the perimeter measure of $E$ is comparable to $\mathcal{H}\vert_{E_\gamma}$ (see~\eqref{eq:def of theta}), we get
\begin{align*}
\mathcal{H}(E_\gamma\cap\partial\Omega)&\le \sum_{j\in\N}\mathcal{H}(E_\gamma\cap\partial\Omega\cap 5\lambda B_j)\\
 &\le C_{\partial\Omega} \sum_{j\in\N}\frac{\mu(5\lambda  B_j)}{5r_j}\\
 &\le C \sum_{j\in\N}\frac{\mu(B_j)}{r_j}
 \le C \sum_{j\in\N} P(E,\Omega\cap \lambda  B_j)\le C P(E,\Omega).
\end{align*}
%
%
Thus by~\eqref{eq:def of theta} again,
\[
  P(E,X)=P(E,\Omega)+P(E,\partial\Omega)\le P(E,\Omega)+C \mathcal{H}(E_\gamma\cap\partial\Omega)
     \le C P(E,\Omega).
\]

Finally, for $u\in \BV(\Omega)$, with $\widehat{u}$ denoting the zero extension of $u$ to $X\setminus\Omega$, by the
coarea formula~\eqref{eq:coarea} we have 
\begin{align*}
  \Vert D\widehat{u}\Vert(X)&=\int_{-\infty}^\infty P(\{\widehat{u}>t\},X)\, dt\\
      &\le C \int_{-\infty}^\infty P(\{u>t\},\Omega)\, dt=C \Vert Du\Vert(\Omega)
\end{align*}
as desired.
\end{proof}

Suppose $\Omega\subset X$ is a bounded open set and $u\in\BV(\Omega)$, 
and suppose that the trace $Tu(x)$ exists for $\mathcal H$-a.e. $x\in\partial\Omega$. We wish to 
have the following $L^1$ estimate for the trace:
\begin{equation}\label{eq:trace L1 estimate}
\int_{\partial\Omega}|Tu|\,d\mathcal H\le C_T\left(\int_{\Omega}|u|\,d\mu+\Vert Du\Vert(\Omega)\right)
\end{equation}
for a constant $C_T>0$ which is independent of $u$. This kind of integral-type inequality is closely 
related to the condition~\eqref{eq:codimension boundary condition}, as we will now show.

\begin{proposition}
Let $\Omega\subset X$ be open and bounded. If \eqref{eq:trace L1 estimate} holds for every $u\in\BV(\Omega)$, then there is a constant
 $C_{\partial\Omega}=C_{\partial\Omega}(C_T,C_d, \text{\rm diam}(\Omega))>0$ such that 
the boundary of $\Omega$ 
satisfies~\eqref{eq:codimension boundary condition} for all 
$x\in\partial\Omega$ and all $r\in (0,2\diam(\Omega))$.
\end{proposition}

\begin{proof}
Pick $x\in\partial\Omega$. For every $i\in \Z$, there exists $r_i\in [2^{i},2^{i+1}]$ such that
\[
P(B(x,r_i),X)\le C_d\frac{\mu(B(x,r_i))}{r_i},
\]
see \cite[Lemma 6.2]{KKST}. We consider $i\in\Z$ for which
$2^{i}\le 4\diam(\Omega)$. Given such an $i$, we choose $u=\chi_{B(x,r_i)\cap\Omega}$.
Then we have
\begin{align*}
\mathcal H(\partial\Omega\cap B(x,r_i))
\le\int_{\partial\Omega}|Tu|\,d\mathcal H
&\le C_T\left(\int_{\Omega}|u|\,d\mu+\Vert Du\Vert(\Omega)\right)\\
&= C_T\left(\mu(B(x,r_i)\cap\Omega)+P(B(x,r_i),\Omega)\right)\\
&\le C_T\left(\mu(B(x,r_i))+P(B(x,r_i),X)\right).
\end{align*}
By the choice of $r_i$, we now have
\begin{align*}
\mathcal H(\partial\Omega\cap B(x,r_i))&\le C_T C_d\left(\mu(B(x,r_i))+\frac{\mu(B(x,r_i))}{r_i}\right)\\
&\le 8C_T C_d[1+\text{diam}(\Omega)]\, \frac{\mu(B(x,r_i))}{r_i}.
\end{align*}
Now by the doubling property of $\mu$ it follows that the above holds for all radii $r\le 2\diam(\Omega)$, by inserting an 
additional factor $C_d^2$ on the right-hand side.
\end{proof}

Now we prove the second of the two main theorems of this paper. 

\begin{theorem}\label{thm:L1-est}
Let $\Omega$ be a bounded domain that supports a 
$(1,1)$-Poincar\'e inequality  and that there is a constant $c_m>0$ such that
the measure density condition~\eqref{eq:measure density condition}
holds for all $x\in\partial\Omega$, $r\in (0, \diam(\Omega))$, and that~\eqref{eq:codimension boundary condition} holds. Then 
there is a constant $C_T=C_T(C_d,C_P,\lambda,\Omega, c_m)>0$ such that~\eqref{eq:trace L1 estimate} holds, that is,
\[
\int_{\partial\Omega}|Tu|\,d\mathcal H\le C_T\left(\int_{\Omega}|u|\,d\mu+\Vert Du\Vert(\Omega)\right)
\]
for all $u\in\BV(\Omega)$.
\end{theorem}

\begin{proof}
By Theorem~\ref{thm:trace theorem}, the 
trace $Tu(x)$ exists for $\mathcal H$-a.e. $x\in\partial\Omega$.
First we show that $Tu$ is a Borel function on $\partial\Omega$, so that the integral 
in~\eqref{eq:trace L1 estimate} makes sense. Since $X$ supports a Poincar\'e inequality, we 
know that a bi-Lipschitz change in the metric results in a geodesic space, i.e. a space where 
every pair of points can be joined by a curve whose length is equal to the distance between the two points, 
see~\cite[Proposition 4.4]{HaKo}. It is easy to see that traces are invariant under this transformation because the 
measure density condition for $\Omega$ holds. It follows from~\cite[Corollary 2.2]{Buc} 
that there are positive constants $C, \delta$ such that whenever $x\in X$, $r>0$, and $\eps>0$, 
\[
\mu(B(x,r[1+\eps])\setminus B(x,r))\le C\mu(B(x,r))\eps^\delta.
\]
It follows from the above condition
and the
absolute continuity of integrals of $u$ that the function
\[
\partial\Omega \ni x\mapsto u_{\Omega\cap B(x,r)}
\]
for any fixed $r>0$ is continuous. Since $Tu$ is the limit of these functions as $r\to 0$, it is a Borel function.

By Proposition~\ref{thm:extension result} 
the zero extension $\widehat{u}=Eu\in\BV(X)$. By \cite[Proposition 5.12]{HKLL} we know that
\[
\int_{\partial\Omega}|(Eu)^{\wedge}|+|(Eu)^{\vee}|\,d\mathcal H\le C\Vert Eu\Vert_{\BV(X)}.
\]
By the measure density condition, it is clear that
\[
(Eu)^{\wedge}(x)\le Tu(x)\le (Eu)^{\vee}(x)
\]
for every $x\in\partial\Omega$ for which the trace $Tu(x)$ exists.
Thus we get by Proposition \ref{thm:extension result}
\[
\int_{\partial\Omega}|Tu|\,d\mathcal H\le \int_{\partial\Omega}|(Eu)^{\wedge}|
   +|(Eu)^{\vee}|\,d\mathcal H\le C\Vert Eu\Vert_{\BV(X)}\le C\Vert u\Vert_{\BV(\Omega)}.
\]
\end{proof}

\begin{example}
Note that once again, if we define $\Omega$ by \eqref{eq:definition of exterior cusp} as an exterior 
cusp in the unweighted space $\R^2$, and consider functions $u_i:=\chi_{B((0,0),1/i)\cap\Omega}$, $i\in\N$, we 
see  that we do not have the $L^1$ estimate~\eqref{eq:trace L1 estimate}, and neither do we have the measure 
density condition \eqref{eq:measure density condition}. If we consider $\overline{\Omega}$ as our metric space, 
we obviously have the measure density condition \eqref{eq:measure density condition}, but then we do not 
have~\eqref{eq:codimension boundary condition}, and again the $L^1$ estimate fails.
\end{example}

On the other hand, the proof of Theorem~\ref{thm:L1-est} does not really require that $\Omega$ 
supports a $(1,1)$-Poincar\'e inequality provided a trace operator exists, as demonstrated in the following example.

\begin{example}
Let $X=\R^2$ and let $\Omega$ be the internal cusp domain given by
\[
 \Omega:=\{x=(x_1,x_2)\in\R^2 :\, |x|<1\text{ and }|x_2|>x_1^2\text{ when }x_1\ge 0\},
\]
equipped with the restriction of the Euclidean metric and Lebesgue measure.
Note that $\Omega$ satisfies the measure density condition and~\eqref{eq:codimension boundary condition},
but $\Omega$ does not support a $(1,1)$-Poincar\'e inequality. However, by Corollary~\ref{cor:exceptional-points} 
we know that functions in $\BV(\Omega)$ do have a trace, and the proof of Theorem \ref{thm:L1-est} shows that
the traces are $\mathcal{H}$-measurable on $\partial\Omega$. Furthermore, $\Omega$
is a BV extension domain, see for example~\cite[Theorem 1.1]{KMS}. 
Now it follows as in the proof of Theorem~\ref{thm:L1-est}
that traces of functions in $\BV(\Omega)$ satisfy an $L^1$ estimate.  Thus the conclusion of Theorem~\ref{thm:L1-est}
is more widely applicable than it first appears.
\end{example}

\section{BV functions with zero trace}

In this section we study functions in $\BV(\Omega)$ that have zero trace on $\partial\Omega$. In considering
Dirichlet problems for a given domain, one is interested in minimizing certain energies of a function from amongst
the class of all functions with the same boundary value; that is, the difference of the test function and the boundary
data should have zero trace. Hence an understanding of $\BV$ functions with zero trace is of importance. 

The following argument can be found in \cite{KKST}, but we give a proof for the reader's convenience. 
Given $f\in\BV(X)$, recall the definition of
the representative $\widetilde{f}$ from~\eqref{eq:pointwise-rep}.

\begin{theorem}\label{thm:zero extension}
Let $\Omega\subset X$ be an open set and let $u\in \BV(\Om)$. Assume either that the space supports a 
strong relative isoperimetric inequality, or that $\mathcal H(\partial\Omega)<\infty$. Then
\begin{equation}\label{eq:weak zero trace condition}
\lim_{r\to 0^+}\frac{1}{\mu(B(x,r))}\int_{B(x,r)\cap\Omega}|u|\,d\mu=0
\end{equation}
for $\cH$-a.e. $x\in\bdy$ if and only if
the zero extension of $u$ into the whole space $X$, denoted by $\widehat{u}$,
is in $\BV(X)$ with $\Vert D\widehat{u}\Vert(X\setminus \Omega)=0$ and 
$\widetilde{\widehat{u}}(x)=0$ for $\mathcal H$-a.e. $x\in\partial\Omega$.

If in 
addition $\Omega$ satisfies the measure density condition~\eqref{eq:measure density condition}, then
the above conditions are also equivalent with the condition of having $Tu(x)=0$ for $\mathcal H$-a.e. $x\in \partial\Omega$.
\end{theorem}

Observe that if $u$ has a trace $Tu$ on $\partial\Omega$, then $u$ necessarily satisfies~\eqref{eq:weak zero trace condition}
and $Tu=0$ $\mathcal{H}$-a.e.~in $\partial\Omega$.

\begin{proof}
First assume that~\eqref{eq:weak zero trace condition} holds for $\mathcal{H}$-a.e.~$x\in\partial\Omega$. 
We denote the level sets of $\widehat{u}$ by
\[
  E_t:=\{x\in X :\, \widehat{u}(x)>t\},\qquad t\in\R.
\]
Fix $t<0$. Then clearly $X\setminus\Om\subset E_t$. If $x\in\partial\Omega$ such 
that~\eqref{eq:weak zero trace condition} holds, then because on $\Omega\setminus E_t$ we have $|u|\ge t$, we get
\[
\limsup_{r\to 0^+}\frac{\mu(B(x,r)\cap\Omega\setminus E_t)}{\mu(B(x,r))}
   \le \limsup_{r\to 0^+} \frac{1}{|t|\mu(B(x,r))}\int_{B(x,r)\cap\Omega}|u|\,d\mu=0.
\]
Hence $E_t$ has density $1$ at $\mathcal H$-a.e. $x\in\partial\Omega$.

Now fix $t>0$. Then
$E_t\subset\Omega$, and if $x\in\partial\Omega$ such that \eqref{eq:weak zero trace condition} 
holds, then because on $E_t$ we have $|u|\ge t$, we get
\[
 \limsup_{r\to 0^+}\frac{\mu(B(x,r)\cap E_t)}{\mu(B(x,r))}
    \le \limsup_{r\to 0^+} \frac{1}{|t|\mu(B(x,r))}\int_{B(x,r)\cap\Omega}|u|\,d\mu=0,
\]
that is, $E_t$ has density 0 at $x$. 

From the above argument we know that 
$\mathcal H(\partial^{*}E_t\cap \partial\Omega)=0$ whenever $t\ne 0$. Since $\widehat{u}$ is a zero 
extension, clearly we have $\partial^*E_t\setminus \overline{\Omega}=\emptyset$ for all $t\in\R$. By the 
coarea formula~\eqref{eq:coarea} we know that $P(E_t,\Omega)<\infty$ for a.e. $t\in\R$. Thus 
for a.e. $t\in\R$, by \eqref{eq:def of theta} we have
\[
\mathcal H(\partial^{*}E_t)=\mathcal H(\partial^{*}E_t\cap\Omega)\le C P(E_t,\Omega)<\infty.
\]
If $\mathcal H(\partial\Omega)<\infty$, by \cite[Proposition 6.3]{KKST} we know that for such $t$, 
necessarily $P(E_t,X)<\infty$. 
The same is true if $X$ supports a strong relative isoperimetric inequality. By~\eqref{eq:def of theta} we then have
\[
P(E_t,X)\le C\mathcal H(\partial^*E_t)\le CP(E_t,\Omega),
\]
and then it follows from the coarea formula \eqref{eq:coarea} that
\[
  \Vert D\widehat{u}\Vert(X)=\int_{-\infty}^\infty P(E_t,X)\, dt\le C \int_{-\infty}^\infty P(E_t,\Omega)\, dt 
      = C\Vert Du\Vert(\Omega)<\infty
\]
and
\[
  \Vert D\widehat{u}\Vert(\bdy)=\int_{-\infty}^\infty P(E_t,\bdy)\, dt
    \le C \int_{-\infty}^\infty \cH(\partial^*E_t\cap \partial\Omega)\, dt = 0.
\]
Hence we also have $\widehat{u}\in\BV(X)$.

Conversely, if we know that the zero extension $\widehat{u}\in\BV(X)$ and 
$\Vert D\widehat{u}\Vert(X\setminus\Omega)=0$, by the decomposition~\eqref{eq:decomposition} 
we know that $\mathcal H(S_{\widehat{u}}\cap\partial\Omega)=0$, and then by the Lebesgue 
point theorem proved in~\cite[Theorem 3.5]{KKST2}, we have
\[
\lim_{r\to 0^+}\frac{1}{\mu(B(x,r))}\int_{B(x,r)\cap\Omega}|u|\,d\mu
   =\lim_{r\to 0^+}\vint{B(x,r)}|\widehat{u}-\widetilde{\widehat{u}}(x)|\,d\mu=0
\]
for $\mathcal H$-a.e. $x\in\partial\Omega$.

The last statement of the theorem is clear.
\end{proof}

To discuss Dirichlet problems in the metric setting, one needs a criterion that determines whether
a BV function has the same boundary values as a given function (the boundary data). One such
criterion is as follows. 

\begin{definition}
We say that $u\in \BV(\Om)$ has \emph{the same boundary values} as $f\in \BV(\Omega)$ if
\begin{equation}\label{eq:boundary value condition}
\lim_{r\to 0^+}\frac{1}{\mu(B(x,r))}\int_{B(x,r)\cap\Omega}|u-f|\,d\mu=0
\end{equation}
for $\mathcal H$-a.e. $x\in\partial\Omega$.
\end{definition}

By Theorem~\ref{thm:zero extension}, 
the zero extension $\widehat{u-f}$ of $u-f$ is in $\BV(X)$ with
$\Vert D(\widehat{u-f})\Vert(\bdy)=0$. Therefore, if $f\in \BV(X)$, then 
the extension $Eu$ of $u$ by $f$ to $X\setminus\Om$ is in $\BV(X)$
with $\Vert D(Eu - f)\Vert(\bdy)=0$. In the Euclidean setting, this is the idea behind the direct method developed by
Giusti, see~\cite[Chapter~14]{Giu84}.

In particular, the condition $Tu(x)=Tf\vert_\Om(x)$ implies $T(u-f\vert_\Om)(x)=0$, which in turn 
implies~\eqref{eq:boundary value condition}, but of course the traces may fail to exist, as was the case in 
Example~\ref{ex:exterior cusp counterexample}.  

\begin{definition}\label{def:zero-bdy}
We define $\BV_0(\Om)$ to be the class of functions 
$u\in \BV(\Om)$ that satisfy~\eqref{eq:weak zero trace condition} for $\cH$-a.e.~$x\in\bdy$.
\end{definition}

Note that in~\cite{HKL} and~\cite{HKLL} a class of BV functions with zero boundary values is defined very 
differently, as functions $u\in\BV(X)$ with $u=0$ outside $\Omega$. For example, if $\Om$ is the unit ball in $\R^n$, 
then $u=\chi_{\Omega}$ is not in $\BV_0(\Om)$,
but does have zero boundary values in the sense of~\cite{HKL} and \cite{HKLL}. In minimization problems, the 
definition for boundary values given in~\cite{HKL} and~\cite{HKLL} can be more convenient, but on the other hand it 
is unclear how much variation measure is concentrated on $\partial\Omega$.
The situation is simpler for Newton-Sobolev functions with zero boundary values 
$N^{1,1}_0(\Omega)$, since whenever $u\in N^{1,1}(X)$ with $u=0$ outside $\Omega$, we 
automatically have $\Vert Du\Vert(X\setminus\Omega)=0$; this follows e.g. from~\cite[Corollary 2.21]{BB}.
We will not pursue the study of Dirichlet problems in the current paper, but will use the above concept of boundary 
values in the study of fine properties of $\BV$ functions. 
To do so, we will need the following lemma.

\begin{lemma}\label{lem:measure in domain}
Let $\Omega\subset X$ be an open set, let $\nu$ be a Radon measure of finite mass on $\Omega$, and define
\[
  A:=\left\{x\in\bdy :\,\limsup_{r\to 0^+}r \frac{\nu(B(x,r)\cap\Om)}{\mu(B(x,r))} > 0\right\}.
\]
Then $\cH(A )=0$.
\end{lemma}
\begin{proof} 
It suffices to show that for each $n\in\N$, the set
\[
  A_n:=\left\{x\in\bdy : \, \limsup_{r\to 0^+}r \frac{\nu(B(x,r)\cap\Om)}{\mu(B(x,r))} > \frac 1n\right\}
\]
has $\cH$-measure zero. Fix $n\in\N$ and $\delta>0$. Then for every $x\in A_n$, we can find a ball of radius $r_x<\delta/5$ such that
\[
 r_x\frac{\nu(B(x,r_x)\cap\Om)}{\mu(B(x,r_x))}>\frac 1n.
\]
By the $5$-covering theorem, we can select a countable pairwise disjoint sub-collection of the collection
$\{B(x,r_x)\}_{x\in A_n}$, denoted by $\{B(x_i,r_i)\}_{i\in\N}$, such that
$A_n\subset \bigcup_{i\in\N} B(x_i,5r_i)$. Now we have 
\begin{align*}
\cH_\delta(A_n) \le \sum_{i\in\N} \frac{\mu(B(x_i,5r_i))}{5r_i}
    &\le \frac{C_d^3}{5}\sum_{i\in\N} \frac{\mu(B(x_i,r_i))}{r_i}\\
    &\le \frac{C_d^3}{5} n \sum_{i\in\N} \nu(B(x_i,r_i)\cap\Omega)\\
    &= C n\, \nu\left(\bigcup_{i\in\N} B(x_i,r_i)\cap\Om\right).
\end{align*}
Given that $r_i<\delta/5$, we know that 
\[
  \bigcup_{i\in\N} B(x_i,r_i)\cap\Om\subset\bigcup_{x\in\bdy}B(x,\delta/5)\cap\Om.
\]
Because $\nu(\Om)<\infty$, it follows that
\begin{align*}
 \cH(A_n)=\lim_{\delta\to 0^+}\cH_\delta(A_n)
 &\le \lim_{\delta\to 0^+} C n\, \nu\left(\bigcup_{x\in\bdy}B(x,\delta/5)\cap\Omega\right)\\
 &= C n\, \nu\left(\bigcap_{\delta>0}\bigcup_{x\in\bdy}B(x,\delta/5)\cap\Omega\right)=C n\, \nu(\emptyset)=0.
\end{align*}
Thus $\mathcal H(A)=0$.
\end{proof}

In most of the remainder of the paper, we will work with Whitney type coverings of open sets.
For the construction of such coverings and their properties, see e.g.~\cite[Theorem 3.1]{BBS07}. Given any 
open set $\Omega\subset X$ and a scale $R>0$, we can choose a Whitney type covering 
$\{B_j=B(x_j,r_j)\}_{j=1}^{\infty}$ of $\Omega$ such that:
\begin{enumerate}
\item for each $j\in\N$,
\begin{equation}\label{eq:radius-dist-est}
r_j:= \min\{\dist(x_j,X\setminus \Omega)/20\lambda,R\},
\end{equation}
\item for each $k\in\N$, the ball $5\lambda B_k$ meets at most $C_0=C_0(C_d,\lambda)$ 
balls $5\lambda B_j$ (that is, a bounded overlap property holds),
\item  
if $5\lambda B_j$ meets $5\lambda B_k$, then $r_j\le 2r_k$,
\item for each $j\in\N$, 
\begin{equation}\label{eq:ball-dist-radius}
  \text{dist}(2B_j,X\setminus\Omega)\ge 18\lambda r_j.
\end{equation}
\end{enumerate}
The last estimate above follows from the first estimate combined with the fact that $\lambda\ge 1$.

Given such a covering of $\Omega$, 
we have a partition of unity $\{\phi_j\}_{j=1}^{\infty}$ subordinate to this covering, that is,  
for each $j\in\N$ the function $\phi_j$ is $C/r_j$-Lipschitz with $\supp(\phi_j)\subset 2B_j$ and $0\le \phi_j\le 1$,
such that $\sum_j\phi_j=1$ on $\Omega$,
see e.g.~\cite[Theorem 3.4]{BBS07}). Finally, we can define a \emph{discrete convolution} $u_W$ of 
$u\in L^1_{\text{\rm loc}}(\Omega)$ with respect to the Whitney type covering by
\[
u_W:=\sum_{j=1}^{\infty}u_{B_j}\phi_j.
\]
In general $u_W\in \liploc(\Omega)\subset L^1_{\text{loc}}(\Omega)$. If $u\in L^1(\Omega)$, then $u_W\in L^1(\Omega)$.

In the next proposition, no assumption is made on the geometry of the open set $\Omega$, 
so that in particular we do not know whether traces of 
BV functions exist. However, we can prove that discrete convolutions have the same boundary values 
as the original function in the sense of~\eqref{eq:boundary value condition}.

\begin{proposition}\label{prop:traces for discrete convolutions}
Let $\Omega\subset X$ be an open set, let $R>0$, 
and let $u\in\BV(\Omega)$. Let $u_W\in\liploc(\Omega)$ be the discrete 
convolution of $u$ with respect to a Whitney type covering $\{B_j=B(x_j,r_j)\}_{j=1}^{\infty}$ of $\Omega$, at  
the scale $R$. Then
\begin{equation}\label{eq:weak trace result for discrete convolutions}
\lim_{r\to 0^+}\frac{1}{\mu(B(x,r))}\int_{B(x,r)\cap\Omega}|u_W-u|\,d\mu=0
\end{equation}
for $\mathcal H$-a.e.~$x\in\partial\Omega$.
\end{proposition}

\begin{proof}
Fix $x\in\partial\Omega$ and take $r>0$. Denote by $I_r$ the set of indices $j\in\N$ for which 
$2B_j\cap B(x,r)\neq \emptyset$, and note that by~\eqref{eq:ball-dist-radius},
\[
r_j\le \frac{\dist(2B_j, X\setminus \Omega)}{18\lambda}\le \frac{r}{18\lambda}
\]
for every $j\in I_r$.
Because $\sum_j\phi_j=\chi_\Omega$, we have
\begin{equation}\label{eq:discrete convolution boundary value calculation}
\begin{split}
\int_{B(x,r)\cap\Omega}|u-u_W|\,d\mu
&= \int_{B(x,r)\cap\Omega}\Big|\sum_{j\in I_r}u\phi_j-\sum_{j\in I_r}u_{B_j}\phi_j\Big|\,d\mu\\
&\le \int_{B(x,r)\cap\Omega}\sum_{j\in I_r}\Big|\phi_j(u-u_{B_j})\Big|\,d\mu\\
&\le \sum_{j\in I_r}\,\int_{2B_j}|u-u_{B_j}|\,d\mu\\
&\le 2C_d \sum_{j\in I_r}\,\int_{2B_j}|u-u_{2B_j}|\,d\mu\\
&\le 4C_d C_P \sum_{j\in I_r}r_j\Vert Du\Vert (2\lambda B_j)\\
&\le C r \Vert Du\Vert (B(x,2r)\cap\Omega).
\end{split}
\end{equation}
In the above, we used the fact that $X$ supports a $(1,1)$-Poincar\'e inequality, and in the last inequality we used the fact that $2\lambda B_j\subset\Omega\cap B(x,2r)$ for all $j\in I_r$ and
 the bounded overlap of the dilated Whitney balls $2\lambda B_j$.
Thus by Lemma \ref{lem:measure in domain}, we have
\[
\lim_{r\to 0^+}\frac{1}{\mu(B(x,r))}\int_{B(x,r)\cap\Omega}|u-u_W|\,d\mu=0
\]
for $\mathcal H$-a.e. $x\in\partial\Omega$.
\end{proof}

Let $\Omega\subset X$ be an open set, $R>0$, and let $u_W$ be the discrete convolution of a function $u\in \BV(\Omega)$ 
with respect to a Whitney type covering $\{B_j\}_{j\in\N}$ of $\Omega$ at scale $R$. Then $u_W$ has an upper gradient
\begin{equation}\label{eq:upper gradient of discrete convolution}
g=C\sum_{j=1}^{\infty}\chi_{B_j}\frac{\Vert Du\Vert(5\lambda B_j)}{\mu(B_j)}
\end{equation}
in $\Omega$, with $C=C(C_d,C_P,\lambda)$, see e.g. the proof of \cite[Proposition 4.1]{KKST2}.
Also, if each $v_i$, $i\in\N$, is a discrete convolution of a function $u\in L^1(\Omega)$ with 
respect to a Whitney type covering of $\Omega$ at scale $1/i$, then
\begin{equation}\label{eq:L1 convergence for discrete convolutions}
v_i\to u \ \textrm { in } L^1(\Omega)\ \textrm { as } 
i\to\infty,
\end{equation}
as seen by the discussion in~\cite[Lemma 5.3]{HKT}.

Combining these facts with Theorem~\ref{thm:zero extension} and 
Proposition~\ref{prop:traces for discrete convolutions}, we obtain the following properties for discrete convolutions.

\begin{corollary}\label{cor:gluing discrete convolutions}
Let $\Omega\subset X$ be an open set, and let $u\in\BV(\Omega)$. Assume either that the space supports a 
strong relative isoperimetric inequality, or that $\mathcal H(\partial \Omega)<\infty$. For each $i\in\N$ let each 
$v_i\in\liploc(\Omega)$ be a discrete convolution of $u$ in $\Omega$ at scale $1/i$. 
Then as $i\to\infty$ we have $v_i\to u$ in $L^1(\Omega)$, 
$\Vert Dv_i\Vert(\Omega)\le\int_\Omega g_{v_i}\, d\mu\le C\Vert Du\Vert(\Omega)$ for some choice of
upper gradients $g_{v_i}$ of $v_i$, and the functions
\[
w_i:=
\begin{cases}
v_i-u &\ \text{in }\Omega,\\
0  &\ \text{in }X\setminus\Omega,
\end{cases}
\]
satisfy $w_i\in \BV_0(\Omega)$ with $\Vert Dw_i\Vert(\partial\Omega)=0$.
\end{corollary}

Because of the lower semicontinuity of the BV-energy norm, it follows from the above corollary that
\[
  \Vert Du\Vert(\Omega)\le \liminf_{i\to\infty}\Vert Dv_i\Vert(\Omega)\le C \Vert Du\Vert(\Omega).
\]
In the Euclidean setting (where the measure is the homogeneous Lebesgue measure), one has recourse to a better
smooth approximation than the discrete convolution presented in this section, namely the smooth convolution 
approximation via smooth compactly supported non-negative radially symmetric convolution kernels; with such a convolution
one has equality in the first inequality above. However, in the metric setting, equality is not guaranteed,
see for example~\cite{HKLL}. 
In order to obtain equality in the above, we need to modify the approximations $v_i$ away from $\partial\Omega$. 
This is the point of the next corollary.

\begin{corollary}\label{cor:gluing discrete convolutions 2}
Let $\Omega\subset X$ be an open set, and let $u\in\BV(\Omega)$. Assume either that the space supports a strong 
relative isoperimetric inequality, or that $\mathcal H(\partial \Omega)<\infty$. Then there exist functions
$\breve{v}_i\in \liploc(\Omega)$, $i\in\N$, with $\breve{v}_i\to u$ in $L^1(\Omega)$, 
$\int_\Omega g_{\breve{v}_i}\, d\mu\to \Vert Du\Vert(\Omega)$ for a choice of upper gradients
$g_{\breve{v}_i}$ of $\breve{v}_i$, and such that the functions
\[
\breve{w}_i:=
\begin{cases}
\breve{v}_i-u &\ \text{in }\Omega,\\
0  &\ \text{in }X\setminus\Omega,
\end{cases}
\]
satisfy $\breve{w}_i\in \BV_0(\Omega)$ with $\Vert D\breve{w}_i\Vert(\partial\Omega)=0$.
\end{corollary}

Note that $\Vert D\breve{v}_i\Vert(\Omega)\le \int_\Omega g_{\breve{v}_i}\, d\mu$, and hence as a consequence
we also have that  $\Vert D\breve{v}_i\Vert(\Omega)\to \Vert Du\Vert(\Omega)$.
When $\breve{v}_i\to u$ in $L^1(\Omega)$ and $\Vert D\breve{v}_i\Vert(\Omega)\to \Vert Du\Vert(\Omega)$, 
we say that the sequence of functions $\breve{v}_i$ converges \emph{strictly} to $u$ in $\BV(\Omega)$. 

\begin{proof} 
For every $\delta>0$, let 
\begin{equation}\label{eq:shrink-domain}
\Omega_{\delta}:=\{y\in\Omega:\,\dist(y,X\setminus\Omega)>\delta\}. 
\end{equation}
Fix $\eps>0$ and $x\in X$, and choose $\delta>0$ such that
\[
\Vert Du\Vert(\Omega\setminus(\Omega_{\delta}\cap B(x,1/\delta)))<\eps.
\]
We will use the function $\eta$ on $X$ given by
\[
\eta(y):=\max\left\{0,1-\frac 4\delta \dist(y,\Omega_{\delta/2}\cap B(x,2/\delta))\right\}
\]
to paste discrete convolutions to good local Lipschitz approximations of $u$ far away from 
$\bdy$. The function $\eta$ is $4/\delta$-Lipschitz.

Let each $v_i\in\liploc(\Omega)$ be a discrete convolution of $u$ in $\Omega$, at scale 
$1/i$. From the definition of the total variation we get a sequence of 
functions $u_i\in\liploc(\Omega)$ with $u_i\to u$ in $L_{\loc}^1(\Omega)$ such that
\[
\int_\Omega g_{u_i}\, d\mu \to\Vert Du\Vert(\Omega)
\]
for some choice of upper gradients $g_{u_i}$ of $u_i$. Now define
\[
\breve{v}_i:=\eta u_i+(1-\eta)v_i,
\]
so that $\breve{v}_i\to u$ in $L^1(\Omega)$, and by a Leibniz rule, see e.g.~\cite[Lemma~2.18]{BB}, 
each $\breve{v}_i$ has an upper gradient
\[
g_i:=g_{\eta}|u_i-v_i|+\eta g_{u_i}+(1-\eta)g_{v_i},
\]
where $g_{\eta}$ and $g_{v_i}$ are upper gradients of $\eta$ and $v_i$, respectively. 
Since we can take $g_{\eta}=0$ outside $\Omega_{\delta/4}\cap B(x,4/\delta)\Subset \Omega$, we have 
$g_{\eta}|u_i-v_i|\to 0$ in $L^1(\Omega)$, and by also  
using~\eqref{eq:upper gradient of discrete convolution}, we get
\begin{align*}
\limsup_{i\to\infty}\int_{\Omega}g_i\,d\mu
&\le \limsup_{i\to\infty}\int_{\Omega}g_{u_i}\,d\mu
    +\limsup_{i\to\infty}\int_{\Omega\setminus (\Omega_{\delta/2}\cap B(x,2/\delta))}g_{v_i}\,d\mu\\
&\le \Vert Du\Vert(\Omega)+C\Vert Du\Vert(\Omega\setminus(\Omega_{\delta}\cap B(x,1/\delta)))\\
&\le \Vert Du\Vert(\Omega)+C\eps.
\end{align*}
The facts that $\breve{w}_i\in \BV(X)$ and $\Vert D\breve{w}_i\Vert(\partial\Omega)=0$ for each $i\in\N$ again follow 
by combining Theorem~\ref{thm:zero extension} and Proposition~\ref{prop:traces for discrete convolutions}. 
By a diagonalization argument, where we also let $\varepsilon\to 0$ (and hence $\delta\to 0$), we complete the proof.
\end{proof}

We can relax the requirement that $X$ supports a strong relative isoperimetric inequality or that 
$\mathcal{H}(\partial\Omega)<\infty$ and still obtain a useful, but slightly weaker, approximation as follows.

\begin{corollary}\label{cor:weak-approx-ver1}
Let $\Omega\subset X$ be an 
open set. Then there is a sequence of open sets $\Omega_j$, $j\in\N$,
with $\Omega_j\Subset \Omega_{j+1}$ for each $j\in\N$ and $\Omega=\bigcup_j\Omega_j$, such that whenever
$u\in \BV(\Omega)$, there is a sequence of functions $u_j\in \BV(\Omega_j)$ such that
\begin{enumerate}
\item $u_j\in\liploc(\Omega_j)$ with 
\[
  \int_{\Omega_j}g_{u_j}\, d\mu\le \Vert Du\Vert(\Omega_j)+2^{-j}
\]
for some choice of upper gradient $g_{u_j}$ of $u_j$,
\item $\Vert u_j-u\Vert_{L^1(\Omega_j)}\le 2^{-j}$,
\item with $\widehat{u_j-u}$ denoting the zero extension of $u_j-u$ to $X\setminus\Omega_j$, 
\[
 \Vert D(\widehat{u_j-u})\Vert(\partial\Omega_j)=0,
\]
\item $\mathcal{H}(\partial\Omega_j)$ is finite for each $j\in\N$.
\end{enumerate}
\end{corollary}

\begin{proof}
For unbounded sets $\Omega$ we can make modifications to the following argument along the lines of the proof of 
Corollary~\ref{cor:gluing discrete convolutions 2}, so we assume now
that $\Omega$ is bounded.

For each $\delta>0$ let $\Omega_\delta$ be as in~\eqref{eq:shrink-domain}.
By the results of~\cite{BH} or~\cite[Proposition 3.1.5]{AT} (see also~\cite[Inequality~(2.6)]{MMS15}), 
$\mathcal{H}(\partial\Omega_\delta)<\infty$ for almost every $\delta>0$. 
For each $j\in\N$ we choose one such $\delta_j\in [2^{-j-1}, 2^{-j}]$.

Since $\mathcal{H}(\partial\Omega_j)$ is finite, we can apply Corollary~\ref{cor:gluing discrete convolutions 2}
to $\Omega$ in order to obtain a function $u_j$ that satisfies the requirements laid out in the statement of the
present corollary. This completes the proof of Corollary~\ref{cor:weak-approx-ver1}.
\end{proof}

In what follows, $\BV_c(\Omega)$ is the collection of all functions $u\in\BV(X)$ with $\text{supt}(u)\Subset\Omega$.
Recall the definition of $\BV_0(\Omega)$ given in Definition~\ref{def:zero-bdy} above.
The class $N^{1,p}_0(\Omega)$ has $N^{1,p}_c(\Omega)$ as a dense subclass whenever $1\le p<\infty$, see
for example~\cite[Theorem~5.45]{BB}.
The following analogous approximation theorem for $\BV_0(\Omega)$ is an application of
Proposition~\ref{prop:traces for discrete convolutions}. Such an approximation is a useful tool in the
study of Dirichlet problems associated with BV functions. The theorem clearly
fails even for smooth bounded Euclidean domains if the notion of $\BV_0(\Omega)$ as given in~\cite{HKL} or~\cite{HKLL}
is considered.

\begin{theorem}\label{thm:density of compactly supported functions revisited-VerI}
Let $\Omega$ be an open set, and suppose that either the space supports a strong relative isoperimetric inequality or 
$\mathcal H(\partial\Omega)<\infty$. Suppose that $u\in \BV_0(\Om)$
(in particular, it is enough to have $u\in \BV(\Omega)$ with $Tu(x)=0$
for $\mathcal H$-a.e. $x\in\partial\Omega$).
Then there exists a sequence $u_k\in \BV_c(\Om)$, $k\in\N$, such that 
$u_k\to u$ in $\BV(\Om)$. 
\end{theorem}

\begin{proof}
First note that functions in $\BV_0(\Omega)$ can be approximated in the $\BV$ norm by bounded functions. To see this,
for each positive integer $n$ set
\[
u_n=\max\{-n,\min\{u,n\}\}.
\] 
Then by the coarea formula and by 
$|u_n|\le |u|$, we see that $u_n\in\BV_0(\Omega)$ with $\lim_nu_n=u$ in $L^1(\Omega)$. For
$t>0$, we see that $u(x)-u_n(x)>t$ if and only if $u(x)>t+n$, and for $t<0$, we have $u(x)-u_n(x)>t$ if and only if
$u(x)>t-n$. Therefore by the coarea formula,
\begin{align*}
 \Vert D(u-u_n)\Vert(\Omega)&=\int_{-\infty}^\infty P(\{u-u_n>t\},\Omega)\, dt\\
        &=\int_{-\infty}^0P(\{u>t-n\},\Omega)\, dt+\int_0^\infty P(\{u>t+n\},\Omega)\, dt\\
        &=\int_{(-\infty,-n)\cup(n,\infty)} P(\{u>t\},\Omega)\, dt\to 0\text{ as }n\to\infty.
\end{align*}
Therefore $u_n\to u$ in $\BV(\Omega)$ as $n\to\infty$. Hence without loss of generality we will assume for the rest
of the proof that $u\in\BV_0(\Omega)$ is bounded.

Fix $\eps>0$. Take a sequence of discrete convolutions $v_i$ of $u$, 
each with respect to a Whitney type covering 
of $\Omega$ at scale $1/i$. By~\eqref{eq:L1 convergence for discrete convolutions} we 
know that $v_i\to u$ in $L^1(\Omega)$ as $i\to\infty$. By 
Proposition~\ref{prop:traces for discrete convolutions} and the fact that $u\in\BV_0(\Omega)$, for each $i\in\N$ we also have
\begin{equation}\label{eq:zero-Sobolev}
\lim_{r\to 0^+}\frac{1}{\mu(B(x,r))}\int_{B(x,r)\cap\Omega}|v_i|\,d\mu=0
\end{equation}
for $\mathcal H$-a.e. $x\in\partial\Omega$.

Let $\Omega_\delta$ be as in~\eqref{eq:shrink-domain} for $\delta>0$, and as in the proof of
Corollary~\ref{cor:gluing discrete convolutions 2}, 
let $\eta_{\delta}\in \Lip_c(\Omega)$ be a $1/\delta$-Lipschitz function such that $0\le \eta_\delta\le 1$ on $X$,
with $\eta_{\delta}=1$ in $\Omega_{2\delta}$ and $\eta_{\delta}=0$ outside 
$\Omega_{\delta}$. Then define
\[
w_{\delta, i}:=\eta_{\delta} u+(1-\eta_{\delta})v_i.
\]
It is clear that $w_{\delta, i}\to u$ in $L^1(\Omega)$ as $i\to\infty$, and by using first the Leibniz rule for BV functions
obtained in~\cite[Corollary~4.2]{KKST2} (unlike for Newtonian functions, using the Leibniz rule for
$\BV$ functions requires that we have bounded functions), and 
then~\eqref{eq:upper gradient of discrete convolution}, we get for large enough $i$
\begin{align*}
&\Vert D(w_{\delta, i}-u)\Vert(\Omega)=\Vert D\big((1-\eta_{\delta})(v_i-u)\big)\Vert(\Omega)\\
&\quad\le C\int_{\Omega_{\delta}\setminus\Omega_{2\delta}}g_{\eta_{\delta}}|v_i-u|\,d\mu
+C\Vert Du\Vert(\Omega\setminus\Omega_{2\delta})
+C\Vert Dv_i\Vert(\Omega\setminus\Omega_{2\delta})\\
&\quad\le\frac C\delta \int_{\Omega_{\delta}\setminus\Omega_{2\delta}}|v_i-u|\,d\mu
+C\Vert Du\Vert(\Omega\setminus\Omega_{3\delta}).
\end{align*}
We can make the second term smaller than $\eps/3$ by taking $\delta$ small enough, and then we can simultaneously
ensure that the 
first term is  smaller than $\eps/3$ and that $\Vert w_{\delta,i}-u\Vert_{L^1(\Omega)}$ is smaller than $\eps/3$
by taking $i$ large enough. We fix the value of $\delta$ so obtained, and then for such large 
enough $i\in\N$, we have
\begin{equation}\label{eq:approximation with half Newtonian function}
\Vert w_{\delta, i}-u\Vert_{\BV(\Omega)}< \eps.
\end{equation}
By~\cite[Theorem~1.1]{KKST} 
and~\eqref{eq:zero-Sobolev}, 
we know that each $v_i$ is in the 
space $N^{1,1}_0(\Omega)$, i.e. the space of Newtonian functions with zero boundary values. 
By~\cite[Theorem~5.45]{BB} we then have for each $i\in\N$ a Lipschitz function 
$\breve{v}_i$, with compact support contained in $\Omega$, such that $\Vert \breve{v}_i-v_i\Vert_{N^{1,1}(\Omega)}<1/i$, 
whence $\Vert \breve{v}_i-v_i\Vert_{\BV(\Omega)}<1/i$. Then we can define for $\delta>0$ and $i\in\N$,
\[
\breve{w}_{\delta, i}:=\eta_{\delta} u+(1-\eta_{\delta})\breve{v}_i.
\]
Clearly $\Vert \breve{w}_{\delta, i}- w_{\delta, i}\Vert_{L^1(\Omega)}\to 0$
as $i\to\infty$, and by the Leibniz rule again,
\begin{align*}
\Vert D(\breve{w}_{\delta, i}-w_{\delta, i})\Vert(\Omega)
&= \Vert D\big((1-\eta_{\delta})(\breve{v}_i-v_i)\big)\Vert(\Omega)\\
&\le C\Vert D(\breve{v}_i-v_i)\Vert(\Omega)+C\int_{\Omega}g_{\eta_{\delta}}|\breve{v}_i-v_i|\,d\mu\to 0
\end{align*}
as $i\to\infty$. By combining this with \eqref{eq:approximation with half Newtonian function}, 
we have for the $\delta>0$ determined earlier and a large enough $i\in\N$ that 
$\Vert \breve{w}_{\delta, i}-u\Vert_{\BV(\Omega)}<\eps$, and furthermore $\breve{w}_{\delta, i}\in \BV_c(\Omega)$.
\end{proof}

\section{Maz'ya-Sobolev inequalities}

In the direct method of the calculus of variations, to obtain regularity of minimizers one needs a Sobolev inequality,
namely, an estimate of the average value (on balls) of a power of a test function in terms of the average value of
its gradient. Such an inequality does not hold for functions in general, as demonstrated by non-zero constant functions. However,
if the test functions of interest vanish on a significantly large set, then such an estimate can be obtained. 
If the largeness of the set is in terms of the measure of the zero set, then such an estimate follows easily from
the Poincar\'e inequality: for $u\in\BV(X)$ and any ball $B=B(x,r)$, we have for $A:=\{y\in B:\,|u(y)|>0\}$,
\[
\left(\,\vint{B}|u|^{Q/(Q-1)}\,d\mu\right)^{(Q-1)/Q}\le Cr\left(1-\left(\frac{\mu(A)}{\mu(B)}\right)^{1/Q}\right)^{-1}
   \frac{\Vert Du\Vert(2\lambda B)}{\mu(2\lambda B)},
\]
where $Q>0$ is the lower mass bound exponent 
from~\eqref{eq:definition of Q}. We refer the interested reader to~\cite[Lemma~2.2]{KKLS} for  proof of this. 
However, in general this is too much to ask for. The more
natural way to measure largeness of the zero set of a function is in terms of the relative capacity of the set. Inequalities
based on such capacitary measures are originally due to Maz'ya~\cite{Maz1, Maz}. For functions with zero trace on the boundary
of $\Omega$, the zero set contains the boundary of $\Omega$, and hence one can see such an inequality as 
associated with $\BV_0(\Omega)$.
Thus we are motivated to study Maz'ya's inequalities for BV functions.


First we define two notions of BV-capacity as follows:

\begin{definition}\label{def:BV capacity}
For an arbitrary set $A\subset X$,
\[
\capa_{\BV}(A):=\inf\Vert u\Vert_{\BV(X)},
\]
where the infimum is taken over functions $u\in\BV(X)$ that satisfy $u\ge 1$ on a neighborhood of $A$.
If $A\subset B$ for a ball $B\subset X$, then the \emph{relative capacity} of $A$ with respect to $2B$ is given by
\[
 \rcapa_{\BV}(A,2B):=\inf\Vert Du\Vert(X)=\inf\Vert Du\Vert(\overline{2B}),
\]
where the infimum is taken over $u\in \BV(X)$ such that $u\ge 1$ on a neighborhood of $A$, with $u=0$ in $X\setminus 2B$.
\end{definition}

Note that if we just required $u\ge 1$ on $A$, we would have $\capa_{\BV}(A)=0$ whenever $\mu(A)=0$. 
For more on $\BV$-capacity, see \cite{HaKi} and \cite{HS}.

For $u\in\BV(X)$, recall the definition of $\widetilde{u}$ from~\eqref{eq:pointwise-rep}
as follows: 
\[
\widetilde{u}(x)=(u^{\wedge}(x)+u^{\vee}(x))/2,\quad x\in X.
\]

By~\cite[Corollary 4.7]{HS} we know that for any compact $K\subset B$,
\[
\rcapa_{\BV} (K,2B)\le C \inf\{\Vert Du\Vert(X):\, \widetilde{u}\ge 1\textrm{ on }K,\ u=0\textrm{ in }X\setminus 2B\}.
\]
Crucially, above we do not require that $\widetilde{u}\ge 1$ in a \emph{neighborhood} of $K$.
By slightly adapting the proof of this result, we obtain an analogous result for $\capa_{\BV}$: for any compact set $K\subset X$,
\begin{equation}\label{eq:capacity of compact sets}
\capa_{\BV}(K)\le C\inf\{\Vert u\Vert_{\BV(X)}:\,\widetilde{u}\ge 1\textrm{ on }K\}.
\end{equation}

By \cite[Corollary 3.8]{HaKi} we know that for any Borel set $A\subset X$,
\begin{equation}\label{eq:capacity and compact  subsets}
\capa_{\BV}(A)=\sup \{\capa_{\BV}(K),\ K\subset A,\ K\textrm{ compact}\}.
\end{equation}
By a direct modification of~\cite[Theorem 3.4]{HaKi} and~\cite[Corollary 3.8]{HaKi}, we 
obtain an analogous result for $\rcapa_{\BV}$: for any Borel set $A\subset B$,
\[
\rcapa_{\BV}(A,2B)= \sup \{\rcapa_{\BV}(K,2B),\ K\subset A,\ K\textrm{ compact}\}.
\]

Adapting the proof of~\cite[Theorem~5.53]{BB} (the result~\cite[Theorem~5.53]{BB} in the
Euclidean setting is originally due to Maz'ya~\cite{Maz1}, see also~\cite[Theorem~10.1.2]{Maz}), 
we now show the following.

\begin{theorem}\label{thm:mazya type inequality}
For $u\in\BV(X)$, let $S:= \{y\in X\setminus S_u:\, \widetilde{u}(y)=0\}$.
Then for every ball $B=B(x,r)\subset X$, we have
\[
\left(\,\vint{2B}|u|^{\tfrac{Q}{Q-1}}\,d\mu\right)^{\tfrac{Q-1}{Q}}\le C\frac{r+1}{\capa_{\BV}(S \cap B)}\Vert Du\Vert(2\lambda B),
\]
and
\[
\left(\,\vint{2B}|u|^{\tfrac{Q}{Q-1}}\,d\mu\right)^{\tfrac{Q-1}{Q}}\le \frac{C}{\text{\rm cap}_{\BV}(S \cap B, 2B)}\Vert Du\Vert(2\lambda B),
\]
with $C=C(C_d,C_P,\lambda)$.
\end{theorem}

\begin{proof}
Set $q:=\tfrac{Q}{Q-1}$.
First we assume that $u$ is bounded, $u\ge 0$,
and that $\Vert u\Vert_{L^1(2B)}>0$.

Let $\eta:X\to[0,1]$ be a $1/r$-Lipschitz function with $\eta=1$ in $B$ and 
$\eta=0$ in $X\setminus 2B$. We use the abbreviation
\[
a:=\left(\,\vint{2B}u^q\,d\mu\right)^{1/q},
\]
and define $v:=\eta(1-u/a)$. Now $v\in\BV(X)$ such that $v=0$ in $X\setminus 2B$ 
and $\widetilde{v}=1$ on $S\cap B$. Pick an arbitrary compact set $K\subset S\cap B$, so 
that in particular $\widetilde{v}=1$ on $K$. By~\eqref{eq:capacity and compact  subsets} and by using the 
Leibniz rule for BV functions, see~\cite[Corollary 4.2]{KKST2}, we get
\begin{equation}\label{eq:relative capacity estimate}
\begin{split}
&\frac 1C \capa_{\BV}(K)
 \le \int_X v\,d\mu+\Vert Dv\Vert(X) \\
&\qquad \le \frac{1}{a}\int_{2B}|u-a|\,d\mu+\frac{C}{a}\left(\Vert Du\Vert(2B)+\int_{2B}|u-a|g_{\eta}\,d\mu\right)\\
&\qquad \le \frac{1+Cr^{-1}}{a}\int_{2B}|u-a|\,d\mu+\frac{C}{a}\Vert Du\Vert(2B).
\end{split}
\end{equation}
To estimate the first term, we write
\begin{equation}\label{eq:relative capacity, using triangle inequality}
\int_{2B}|u-a|\,d\mu\le \int_{2B}|u-u_{2B}|\,d\mu+|u_{2B}-a|\,\mu(2B).
\end{equation}
Here the first term can be estimated using the $(1,1)$-Poincar\'e inequality, whereas for the second term we have
\begin{align*}
|a-u_{2B}|\,&\mu(2B)^{1/q}
=|\Vert u\Vert_{L^q(2B)}-\Vert u_{2B}\Vert_{L^q(2B)}|\\
&\qquad\qquad\le \Vert u-u_{2B}\Vert_{L^q(2B)}\\
&\qquad\qquad= \left(\,\vint{2B}|u-u_{2B}|^q\,d\mu\right)^{1/q}\mu(2B)^{1/q}\\
&\qquad\qquad= Cr\frac{\Vert Du\Vert (2\lambda B)}{\mu(2B)}\mu(2B)^{1/q}.
\end{align*}
Inserting this into \eqref{eq:relative capacity, using triangle inequality}, we get
\[
\int_{2B}|u-a|\,d\mu\le Cr\Vert Du\Vert (2\lambda B).
\]
Inserting this into \eqref{eq:relative capacity estimate}, we then get
\[
\capa_{\BV}(K)\le C\frac{r+1}{a}\Vert Du\Vert(2\lambda B).
\]
Recalling the definition of $a$ and using \eqref{eq:capacity and compact  subsets}, this implies
\[
\left(\,\vint{2B}|u|^q\,d\mu\right)^{1/q}\le C\frac{r+1}{\capa_{\BV}(S\cap B)}\Vert Du\Vert(2\lambda B).
\]
A modification of the above argument (by dropping the term $\int_X v\, d\mu$
from~\eqref{eq:relative capacity estimate}) yields the second inequality claimed in the theorem for bounded
non-negative functions. For bounded $\BV$ functions that are allowed to take negative values, we can replace $u$ with
$|u|$ to obtain the desired inequalities by noting that for any $A\subset X$,
\[
\Vert D|u|\Vert(A)\le \Vert Du\Vert(A).
\]
Finally, for $\BV$ functions that may also be unbounded, as at the beginning of the proof of 
Theorem~\ref{thm:density of compactly supported functions revisited-VerI} we approximate
$u$ by bounded functions $u_n$ in the $\BV(X)$-norm and note that 
\[
\lim_{n\to\infty}\int_{2B}|u_n|^q\, d\mu=\int_{2B}|u|^q\, d\mu
\]
and 
\[
\lim_{n\to\infty}\Vert Du_n\Vert(2\lambda B)=\Vert Du\Vert(2\lambda B).
\]
Thus we obtain the desired capacitary estimates for all $u\in\BV(X)$.
\end{proof}

%
%

\end{document}